\documentclass[11pt]{article}

\usepackage{amsthm,amsmath,amsfonts,amssymb,mathrsfs}
\usepackage[numbers]{natbib}
\usepackage[colorlinks,citecolor=blue,urlcolor=blue]{hyperref}
\usepackage{graphicx,color}

\theoremstyle{plain}
\newtheorem{theorem}{Theorem}[section]

\newtheorem{lemma}[theorem]{Lemma}
\newtheorem{proposition}[theorem]{Proposition}

\theoremstyle{definition}

\theoremstyle{remark}
\newtheorem{remark}[theorem]{Remark}

\numberwithin{equation}{section}
\date{}

\title{Convergence of the derivative martingale for the branching random walk in time-inhomogeneous  random environment \thanks{This work was supported in part by  NSFC (No.~11971062), and the National Key Research and Development Program of China (No.~2020YFA0712900).}}
\author{Wenming Hong \thanks{School of Mathematical Sciences \& Laboratory of Mathematics and Complex Systems, Beijing Normal University, Beijing 100875, P.R. China. Email: wmhong@bnu.edu.cn} ~and~ Shengli Liang \thanks{School of Mathematical Sciences \& Laboratory of Mathematics and Complex Systems, Beijing Normal University, Beijing 100875, P.R. China. Email: liangshengli@mail.bnu.edu.cn}}

\begin{document}
	
	\maketitle	
	\begin{center}
		\begin{minipage}{12cm}
			\begin{center}\textbf{Abstract}\end{center}
			Consider a branching random walk on the real line with a random environment in time (BRWRE). A necessary and sufficient condition for the non-triviality of the limit of the derivative martingale is formulated. To this end,  we investigate the  random walk in time-inhomogeneous random  environment (RWRE), which related the BRWRE by the many-to-one formula. The key step is to figure out Tanaka's decomposition for the RWRE conditioned to stay non-negative (or above a line), which is interesting itself as well.
		 
			\bigskip
			
			\mbox{}\textbf{Keywords}: Branching random walk; random environment; derivative martingale; quenched harmonic function; random walk conditioned to stay non-negative, Tanaka's decomposition. \\
			\mbox{}\textbf{Mathematics Subject Classification}: Primary 60J80; 60K37; secondary 60G42.
			
		\end{minipage}
	\end{center}


\section{Introduction and main result}\label{section 1}

Consider a discrete-time branching random walk on $\mathbb{R}$ in the random environment (BRWRE). The \textit{random environment} is represented by a sequence of random variables $\xi=\left(\xi_{n}, n\geq 1\right)$ which defined on some probability space $\left(\Omega, \mathscr{A}, \mathbf{P}\right)$. We assume throughout that $\left(\xi_{n}, n\geq 1\right)$ are independent and identically distributed (i.i.d.) random variables. Each realization of $\xi_{n}$ corresponding a point process law $\mathcal{L}_{n}=\mathcal{L}\left(\xi_{n}\right)$. Given the environment, the \textit{time-inhomogeneous branching random walk} is described as follows. It starts at time $0$ with an initial particle (denote by $\varnothing$) positioned at the origin. This particle dies at time $1$ and gives birth to a random number of children who form the first generation and whose positions are given by a point process $L_{1}$ with law $\mathcal{L}_{1}$. For any integer $n\geq1$, each particle alive at generation $n$ dies at time $n+1$ and gives birth independent of all others to its own children who are in the $(n+1)$-th generation and are positioned (with respect to their parent) according to the point process $L_{n+1}$ with law $\mathcal{L}_{n+1}$. All particles behave independently conditioned on the environment $\xi$. The process goes on as described above if there are particles alive. We denote by $\mathbb{T}$ the genealogical tree of the process. For a given vertex $u\in \mathbb{T}$, we denote by $V\left(u\right)\in\mathbb{R}$ its position and $|u|$ its generation. We write $u_{i}$ $\left(0\leq i\leq|u|\right)$ for its ancestor in the $i$-th generation (with the convention that $u_{0}:=\varnothing$ and $u_{|u|}=u$). Given a realization of $\xi$, we write $\mathbb{P}_{\xi}$ for the conditional (or quenched) probability and $\mathbb{E}_{\xi}$ for the corresponding expectation. The joint (or annealed) probability of the environment and the branching random walk is written $\mathbb{P}:=\mathbb{P}_{\xi}\otimes\mathbf{P}$, that is, $$\mathbb{P}(\cdot)=\int_{\Omega}\mathbb{P}_{\xi}(\cdot)\, d\mathbf{P},$$ with the corresponding expectation $\mathbb{E}$.

This model was first introduced by Biggins and Kyprianou \cite{BK04}. Recently, some results for the homogeneous branching random walk have been extended to the BRWRE. Huang and Liu \cite{HL14} proved a law of large numbers for the maximal position and large deviation principles for the counting measure of the process. Gao, Liu and Wang \cite{GLW14} obtained the central limit theorem. Wang and Huang \cite{WH17} considered the $L^{p}$ convergence rate of additive martingale and moderate deviations principles for the counting measure. Mallein and Mi\l{}o{\'s} \cite{MM19} investigated the second order asymptotic behavior of maximal displacement. Also, there are many authors who focus on other kind of random environments. For example, Greven and den Hollander \cite{GdH92} considered the branching random walk with the reproduction law of the particles depending on their location. Yoshida \cite{Yos08} and Hu and Yoshida \cite{HY09} investigated the branching random walk with space-time i.i.d. offspring distributions.

In this paper, we consider the limit of the derivative martingale for the BRWRE, which proved to play an important role in the convergence of both minimal position and additive martingale for the classical branching random walk, see A\"{i}d\'{e}kon \cite{Aid13}, A\"{i}d{\'e}kon and Shi \cite{AS14}, respectively. Also, one original motivation for our result comes from the study of the Seneta-Heyde norming  of the additive martingale for this model.

For each $n\geq1$, $t\in\mathbb{R}$, we introduce the $\log$-Laplace transform of the point process  $L_{n}$ as follows:
$$\Psi_{n}\left(t\right):=\log\mathbb{E}_{\xi}\left[\int_{\mathbb{R}}e^{-tx}L_{n}\left(dx\right)\right]=\log\mathbb{E}_{\xi}\left[\sum_{x\in L_{n}}e^{-tx}\right].$$
The \textit{additive martingale} is defined as:
$$W_{n}(t):=\sum_{|u|=n}e^{-tV(u)-\sum_{i=1}^{n}\Psi_{i}(t)}.$$
Let $\mathscr{F}_{n}:=\sigma\left(\xi_{1},\xi_{2},\cdots,\left(u, V(u)\right), |u|\leq n\right)$. It is well known that for each $t$ fixed, $\left(W_{n}(t),n\geq0\right)$ forms a non-negative martingale with respect to the filtration $\left(\mathscr{F}_{n},n\geq0\right)$ under both laws $\mathbb{P}_{\xi}$ and $\mathbb{P}$. By the martingale convergence theorem, $W_{n}(t)$ converges almost surely (a.s.) to a non-negative limit. In the deterministic environment case, Biggins \cite{Big77} gave a necessary and sufficient condition for the $L^{1}$-convergence of $W_{n}(t)$, we refer to Lyons \cite{Lyo97} for a simple probabilistic proof based on the spinal decomposition. Later, Biggins and Kyprianou \cite{BK04} extended this to the random environment case.

To ensure the non-extinction and non-triviality of the BRWRE, we assume that
\begin{equation}\label{ass1}
	\begin{aligned}
		&\mathbb{P}_{\xi}\left(\sum_{|u|=1}1\geq1\right)=1, ~~\mathbf{P}\text{-a.s.}, ~~~~\mathbf{P}\left[\mathbb{P}_{\xi}\left(\sum_{|u|=1}1>1\right)>0\right]>0,\\ &\mathbb{E}_{\xi}\left[\sum_{|u|=1}\mathbf{1}_{\left\{V(u)>0\right\}}e^{-V(u)}\right]>0, ~~\mathbf{P}\text{-a.s.}.
	\end{aligned}
\end{equation}

We consider the \textit{boundary case} (in the quenched sense) in this paper, that is,
\begin{equation}\label{boundary condition}
	\log\mathbb{E}_{\xi}\left[\sum_{|u|=1}e^{-V(u)}\right]=\mathbb{E}_{\xi}\left[\sum_{|u|=1}V(u)e^{-V(u)}\right]=0, ~~\mathbf{P}\text{-a.s.}.
\end{equation}
In fact, if we assume that there exists $t^{*}>0$ such that $\mathbf{P}${-a.s.} $\Psi_{1}$ is differentiable at point $t^{*}$ and $\Psi_{1}(t^*)=t^*\Psi^{'}_{1}(t^*)$, then without loss of generality we can assume that $t=1$ and $\Psi_{1}(1)=\Psi^{'}_{1}(1)=0, \mathbf{P}\text{-a.s.}$ For a general case, we can construct a new BRWRE with position replaced by $\tilde{V}(u):=t^*V(u)+\sum_{i=1}^{|u|}\Psi_{i}(t^*)$, $u\in\mathbb{T}$, the $\log$-Laplace transform of this new process satisfies $\tilde{\Psi}_{1}(1)=\tilde{\Psi}^{'}_{1}(1)=0, \mathbf{P}\text{-a.s.}$.

We take interest in the \textit{derivative martingale}, defined by
$$D_{n}:=\sum_{|u|=n}V(u)e^{-V(u)},~~n\geq 0.$$
It is easy to show that, in the boundary case, $\left(D_{n},n\geq0\right)$ is a signed martingale with respect to the filtration $\left(\mathscr{F}_{n},n\geq0\right)$ under both laws $\mathbb{P}_{\xi}$ and $\mathbb{P}$. For the branching random walk with constant environment, the derivative martingale has been studied in different contexts. From the perspective of smoothing transformation in the sense of Durrett and Liggett \cite{DL83} and Liu \cite{Liu98}, the limit of the derivative martingale serves as a fixed point of the smoothing transformation and the existence, uniqueness, asymptotic behavior of such fixed point has been investigated in \cite{BK05,Kyp98,Liu00}. In \cite{BK04}, Biggins and Kyprianou derived the sufficient condition for the non-triviality (and triviality) of the limit of the derivative martingale. Later, A\"{i}d\'{e}kon \cite{Aid13} gave the optimal condition for the non-triviality which proved to be necessary by Chen \cite{Che15}. For the branching Brownian motion, the necessary and sufficient condition for the non-degeneracy of the limit of the derivative martingale was given by Yang and Ren \cite{YR11}. Recently, Mallein and Shi \cite{MS22} obtained the necessary and sufficient condition for branching L\'{e}vy process.

In addition, we assume that there exists $\delta>0$ such that
\begin{equation}\label{2 more moment}
	\mathbb{E}\left[\sum_{|u|=1}V(u)^{2+\delta}e^{-V(u)}\right]<\infty,
\end{equation}
The assumption $\left(\ref{2 more moment}\right)$ ensures us to prove the existence and some useful asymptotic behaviours of the quenched harmonic function in Section \ref{section 2}.

The main result of this paper is to prove the existence of the limit of derivative martingale for the BRWRE and obtain a necessary and sufficient condition for the non-degeneracy of the limit, which is stated as follows.

\begin{theorem}\label{main result}
	Under the assumptions $\left(\ref{ass1}\right)$, $\left(\ref{boundary condition}\right)$ and $\left(\ref{2 more moment}\right)$, we have
	
	$\left(1\right)$ The derivative martingale $\left(D_{n},n\geq0\right)$ converges almost surely to a non-negative finite limit which we denote by $D_{\infty}$, i.e., $$\lim_{n\to\infty}D_{n}=D_{\infty}\geq0, ~~\mathbb{P}\text{-a.s.}.$$
	
	$\left(2\right)$ For almost all $\xi$, $D_{\infty}$ is non-triviality if and only if
	\begin{equation}\label{iff}
		\mathbb{E}\left[Y\log^2_{+}Y+Z\log_{+}Z\right]<\infty,
	\end{equation}
    more precisely,
    \begin{equation}\nonumber
    	\begin{aligned}
    		&\mathbb{P}_{\xi}\left(D_{\infty}>0\right)>0, ~~\mathbf{P}\text{-a.s.}, ~~\Longleftrightarrow~~ \mathbb{E}\left[Y\log^2_{+}Y+Z\log_{+}Z\right]<\infty,\\
    		&\mathbb{P}_{\xi}\left(D_{\infty}=0\right)=1, ~~\mathbf{P}\text{-a.s.}, ~~\Longleftrightarrow~~ \mathbb{E}\left[Y\log^2_{+}Y+Z\log_{+}Z\right]=\infty,
    	\end{aligned}
    \end{equation}
    where $\log_{+}x:=\max\left\{0, \log x\right\}$ and $\log^2_{+}x:=\left(\log_{+}x\right)^2$ for any $x\geq0$, and
    $$Y:=\sum_{|u|=1}e^{-V(u)}, ~~Z:=\sum_{|u|=1}V(u)e^{-V(u)}\mathbf{1}_{\left\{V(u)\geq 0\right\}}.$$
\end{theorem}

\

The idea to prove the Theorem is followed the  general argument of Biggins and Kyprianou \cite{BK04}. Note that for the constant environment situation (A\"{i}d\'{e}kon \cite{Aid13} and  Chen \cite{Che15}), 
a basic tool is Tanaka's decomposition for the random walk conditioned to stay non-negative. Here  we should deal with the {\it random environment}. To this end,  we investigate the  random walk in time-inhomogeneous random  environment (RWRE), which related the BRWRE by the many-to-one formula. Based on the quenched harmonic function (\cite{HL22}) for the RWRE, we  figure out quenched Tanaka's decomposition for the RWRE conditioned to stay non-negative (Proposition \ref{Tanaka decomposition} and \ref{Annealed excursion distribution}), which is  a novelty of this paper and is interesting itself as well.

Let us describe briefly the proof of Theorem \ref{main result}. To prove the a.s. convergence of the derivative martingale $D_n$, we introduce the truncated martingale $D^{(\beta)}_n$ by formulating a quenched harmonic function for the associated random walk, and use $D^{(\beta)}_n$ to approach $D_n$. To prove the necessary and sufficient condition for the non-degeneracy of the limit $D_{\infty}$, we adapt the general argument of Biggins and Kyprianou \cite{BK04}, using truncated martingale to define a new probability, under which the branching random walk is characterized by a spinal decomposition (Proposition \ref{spinal decomposition} and \ref{law of the spine}). By means of the spinal decomposition, we can give the criterion conditions for the $L^1$-convergence of $D^{(\beta)}_n$ or the degeneracy of the limit $D^{(\beta)}_{\infty}$ (Proposition \ref{BK th 2.1}). Then by Lemma \ref{connection}, these conditions are equivalent to the non-triviality or triviality of the limit $D_{\infty}$ of the derivative martingale $D_n$. Thus, it boils down to checking the conditions of Proposition \ref{BK th 2.1}. For the a.s. convergence of the random series in Proposition \ref{BK th 2.1} (1), we show that its expectation is finite under $\left(\ref{iff}\right)$. For the a.s. divergence of the random series in Proposition \ref{BK th 2.1} (2), we prove it by a equivalent integral condition (Proposition \ref{prop 0-1 law}), which holds when $\left(\ref{iff}\right)$ fails.

The rest of this paper is organized as follows. In Section \ref{section 2}, we introduce a quenched harmonic function which is used to constructed the random walk conditioned to stay above a line. Then, we give a version of Tanaka's decomposition for the random walk conditioned to stay non-negative in our setting, by which an equivalent integral condition for the a.s. divergence of random series associated with conditioned random walk is proved. In Section \ref{section 3}, we use a truncated derivative martingale to make a change of measure and give a spinal decomposition of time-inhomogeneous branching random walk, the details are given in Appendix. Finally, we obtain a necessary and sufficient condition for the non-trivial limit of the derivative martingale in Section \ref{section 4}.

Throughout the paper, we denote by $\left(c_{i},i\geq0\right)$ the positive constants and $c_{i}(\beta)$ the positive constant depending on $\beta$. The indicator function on the
event $A$ is denoted by $\mathbf{1}_{A}$. We use $x_{n}\sim y_{n} (n\to\infty)$ to denote $\lim_{n\to\infty}\frac{x_{n}}{y_{n}}=1$, and when $x_{n}$ and $y_{n}$ are random variables, the limit holds in the sense of a.s. convergence. Write for $x\in\mathbb{R}\cup\{\infty\}\cup\{-\infty\}$, $x_{+}:=\max\{x,0\}$. We also adopt the notations $\sum_{\emptyset}\left(\cdots\right):=0$ and $\prod_{\emptyset}\left(\cdots\right):=1$.


\section{Quenched harmonic function and conditioned random walk}\label{section 2}

We introduce in this section the many-to-one lemma that links BRWRE with RWRE. Then, based on the quenched harmonic function (\cite{HL22}) for the RWRE, we define the law of the random walk conditioned to stay above a line. After discussing the relationship between the quenched probability $\mathbb{P}^{+,(\beta)}_{\xi}\left(x;dy\right)$ and the annealed renewal measure $\mathcal{R}(dy)$, we give a version of Tanaka's decomposition. And as a consequence, an equivalent integral condition for the a.s. divergence of random series about conditioned random walk is proved.

\subsection{The many-to-one lemma}

The well-known many-to-one lemma is a powerful tool in the study of branching random walk, see Shi \cite{Shi15} and reference therein. In this paper, we need a time-inhomogeneous version of this lemma. For all $n\geq1$, we define the probability measure $\mu_{n}$ on $\mathbb{R}$ by
$$\mu_{n}\left(B\right):=\mathbb{E}_{\xi}\left[\sum_{x\in L_{n}}\mathbf{1}_{\left\{x\in B\right\}}e^{-x}\right], ~~\forall B\in\mathscr{B}(\mathbb{R}).$$
Note that $\mu_{n}\left(B\right)$ is a random variable depending on $\xi$. Up to a possible enlargement of the probability space, we define a sequence $\left(X_{n}, n\geq1\right)$ of independent random variables, where $X_{n}$ has law $\mu_{n}$. Let $S_{n}:=S_{0}+\sum_{i=1}^{n}X_{i}$. The process $\left(S_{n},n\geq0\right)$ is a random walk with random environment in time. For our convenience, $\mathbb{P}_{\xi}$ also stands for the joint law of the BRWRE and the RWRE, given the environment $\xi$. If we emphasize that the process starts from $a\in\mathbb{R}$, this law will denote by $\mathbb{P}_{\xi,a}$ and $\mathbb{P}_{\xi}:=\mathbb{P}_{\xi,0}$. The following time-inhomogeneous many-to-one lemma can be found in Lemma 2.2 of Mallein \cite{Mal15}.

\begin{lemma}[Many-to-one]
	For all $n\geq1$ and any measurable function $f:\mathbb{R}^{n}\to\mathbb{R}_{+}$, we have
	\begin{equation}\label{many-to-one}
		\mathbb{E}_{\xi, a}\left[\sum_{|u|=n}f\left(V(u_{1}), \cdots, V(u_{n})\right)\right]=\mathbb{E}_{\xi, a}\left[e^{S_{n}-a}f\left(S_{1}, \cdots, S_{n}\right)\right], ~~\mathbf{P}\text{-a.s.}
	\end{equation}
    with $\mathbb{P}_{\xi, a}\left(S_{0}=a\right)=1, ~~\mathbf{P}\text{-a.s.}$.
\end{lemma}

\subsection{Quenched harmonic function}

In this subsection, we introduce the quenched harmonic function which will be used to construct the random walk conditioned to stay in a given interval.

It follows from $\left(\ref{ass1}\right)$, $\left(\ref{boundary condition}\right)$, $\left(\ref{2 more moment}\right)$ and $\left(\ref{many-to-one}\right)$ that
\begin{equation}\label{condition for quenched harmonic function}
	\begin{aligned}
		&\mathbb{P}_{\xi}\left(S_{1}>0\right)=\mathbb{E}_{\xi}\left[\sum_{|u|=1}\mathbf{1}_{\left\{V(u)>0\right\}}e^{-V(u)}\right]>0, ~~\mathbf{P}\text{-a.s.},\\
		&\mathbb{E}_{\xi}(S_{1})=\mathbb{E}_{\xi}\left[\sum_{|u|=1}V(u)e^{-V(u)}\right]=0,~~\mathbf{P}\text{-a.s.},\\
		&\mathbb{E}(S^{2+\delta}_{1})=\mathbb{E}\left[\sum_{|u|=1}V(u)^{2+\delta}e^{-V(u)}\right]<\infty.
	\end{aligned}
\end{equation}
Under $\left(\ref{condition for quenched harmonic function}\right)$, we can formulate the quenched harmonic function as below.

Let $y\geq0$, denote by $\tau_{y}$ the first moment when $\{S_n\}$ enters the interval $(-\infty, -y)$: $$\tau_{y}:=\inf\left\{n\geq1: y+S_n<0\right\}.$$
Define $$U_n(\xi,y):=\mathbb{E}_\xi\left((y+S_n)\mathbf{1}_{\left\{\tau_{y}>n\right\}}\right).$$ Let $\theta$ be the shift operator, i.e. $\theta\xi:=(\xi_2,\xi_3,\cdots)$. For $n\geq1$, $\theta^n\xi:=\theta(\theta^{n-1}\xi)$ with the
convention that $\theta^0\xi:=\xi$. The following proposition (see Proposition 1.1, Theorem 1.2 and Corollary 1.3 of Hong and Liang \cite{HL22}) proves the existence and asymptotic behaviour of a positive quenched harmonic function.

\begin{proposition}\label{harmonic function}
	For almost all $\xi$, we have the following statements.
	
	$\left(1\right)$ There exists a random variable $U(\xi,y)$ such that $$\lim\limits_{n\to \infty}U_{n}(\xi,y)=U(\xi,y):=-\mathbb{E}_\xi(S_{\tau_y})<\infty.$$
	
	$\left(2\right)$ $U(\xi,y)$ satisfies the quenched harmonic property:  $$U(\xi,y)=\mathbb{E}_{\xi}\left[U\left(\theta \xi,y+S_{1}\right)\mathbf{1}_{\left\{\tau_{y}>1\right\}}\right].$$
	
	$\left(3\right)$ $\left(U(\theta^n\xi,y+S_n)1_{\{\tau_{y}>n\}},n\geq1\right)$ is a martingale under $\mathbb{P}_\xi.$
	
	$\left(4\right)$ $U(\xi,y)$ is positive, non-decreasing in $y$, $U(\xi,y)\geq y$ and
		\begin{equation}\label{asymptotic behaviour of harmonic function 1}
			\lim\limits_{y\to\infty}\frac{U(\xi,y)}{y}=1.
		\end{equation}
	
	$\left(5\right)$ For any $y_n\geq0$ with $y_n\to \infty$ as $n\to \infty$, 
	    \begin{equation}\label{asymptotic behaviour of harmonic function 2}
	    	\lim\limits_{n\to\infty}\frac{U(\theta^n\xi,y_n)}{y_n}=1.
	    \end{equation}
\end{proposition}

\begin{remark}\label{Rem1}
	For the case of classical random walk, one important tool to analyse the behavior of this process conditioned to stay non-negative is the well-known Wiener-Hopf factorisation, we refer to the standard book of Feller \cite{Fel71}. For any oscillating random walk, the renewal function associated with ladder heights process is harmonic, see $\left(\ref{usual harmonic}\right)$. These methods essentially rely on the so-called duality principle which, unfortunately, failed in our setting, since the random walk is time-inhomogeneous given the environment. In  \cite{HL22}, the quenched harmonic function is obtained by strong approximation.
\end{remark}

\subsection{Random walk conditioned to stay above a line: quenched and annealed}

\subsubsection{Quenched probability $\mathbb{P}^{+,(\beta)}_{\xi}\left(x;dy\right)$ and annealed renewal measure $\mathcal{R}(dy)$  }

For any fixed $\beta\geq0$, we introduce the quenched random walk conditioned to stay above $-\beta$ for almost all $\xi$, which denote by $\zeta^{(\beta)}_{n}$, in the sense of Doob's \textit{h}-transform. By Proposition \ref{harmonic function} $\left(3\right)$, for any $n\geq1$ and $B\in\mathscr{B}(\mathbb{R})$, we can define the law of $\zeta^{(\beta)}_{n}$ by
\begin{equation}\label{Prob of conditioned rw}
	\mathbb{P}_{\xi,a}(\zeta^{(\beta)}_{0}=a):=1, ~~\mathbb{P}_{\xi,a}(\zeta^{(\beta)}_{n}\in B):=\frac{\mathbb{E}_{\xi,a}\left(U(\theta^n \xi, S_n+\beta)\mathbf{1}_{\{\tau_{\beta}>n\}}\mathbf{1}_{\left\{S_{n}\in B\right\}}\right)}{U(\xi,a+\beta)}.
\end{equation}
The process $\left(\zeta^{(\beta)}_{n},n\geq0\right)$ is called a \textit{random walk in time random environment conditioned to stay above $-\beta$}. Indeed, $\left(\zeta^{(\beta)}_{n},n\geq0\right)$ is a Markov chain with state space $[-\beta,\infty)$ and the transition kernel given by
$$\mathbb{P}^{+,(\beta)}_{\xi}\left(x;dy\right):=\frac{U(\theta\xi,y+\beta)\mathbf{1}_{\left\{y\geq -\beta\right\}}}{U(\xi,x+\beta)}\mathbb{P}_{\xi,x}\left(S_{1}\in dy\right).$$

On the other hand, $\left(S_{n},n\geq0\right)$ is a simple random walk under the annealed law $\mathbb{P}$ because the environment are i.i.d..  Recall that given environment $\xi$, $X_{n}$ has law $\mu_{n}$, let $\mathbf{E}\left(\mu_{n}\right)$ be the annealed probability of $\mu_{n}$ and $\mu^{\infty}:=\prod_{n=1}^{\infty}\mathbf{E}\left(\mu_{n}\right)$ be the product probability, denote by $\mathbb{E}_{\mu^{\infty}}$ the corresponding expectation, then $\left(S_{n},n\geq0\right)$ is an usual random walk under $\mu^{\infty}$. In the case we treat the annealed random walk (that is $\left(S_{n},n\geq0\right)$ under $\mathbb{P}_{a}$), we shall identify the law $\mathbb{P}_{a}$ with law $\mu^{\infty}_{a}$.

In what follows, we state the usual construction of classical random walk conditioned to stay above a given value, which is indicated in Remark \ref{Rem1}. Define the strict descending ladder epochs of random walk $(S_{n},n\geq0)$ as
$$\gamma_{0}:=0,~ \gamma_{k+1}:=\inf\left\{n>\gamma_{k}:S_{n}< S_{\gamma_{k}}\right\},~ k\geq0.$$
Let $R^{-}$ be the function associated with $(S_{n},n\geq0)$ that defined by
$$R^{-}(0):=1, ~R^{-}(x):=\sum_{k=0}^{\infty}\mu^{\infty}(S_{\gamma_{k}}\geq -x),~x>0.$$
$R^{-}(x)$ is a renewal function of the ladder heights $(-S_{\gamma_{k}})$. Let $\mathcal{R}^{-}(dx)$ be the corresponding renewal measure. By the renewal theorem (cf. \cite{Fel71}, Chapter XI, Section 1), in our setting, we have
\begin{equation}\nonumber
	\lim_{x\to \infty}\frac{R^{-}(x)}{x}=c_{0}\in(0,\infty).
\end{equation}
The function $R^{-}(x)$ satisfies (cf. Lemma 1 of \cite{Tan89})
\begin{equation}\label{usual harmonic}
	\mu^{\infty}\left[R^{-}(x+X_{1})1_{\left\{x+X_{1}\geq 0\right\}}\right]=R^{-}(x),~~\text{for}~ x\geq0.
\end{equation}
From $\left(\ref{usual harmonic}\right)$, it follows that $(R^{-}(S_{n}+\beta)1_{\left\{\underline{S}_{n}\geq -\beta\right\}},n\geq1)$ is a martingale under $\mu^{\infty}$, where $\underline{S}_{n}:=\min\left\{S_{0},\cdots,S_{n}\right\}$. Thus, we can construct the random walk conditioned to stay above $-\beta$, which denote by $\eta^{(\beta)}_{n}$, that is, for any $B\in\mathscr{B}(\mathbb{R})$,
\begin{equation}\nonumber
	\mu^{\infty}_{a}(\eta^{(\beta)}_{n}\in B):=\frac{\mu^{\infty}_{a}[R^{-}(S_{n}+\beta)\mathbf{1}_{\left\{\underline{S}_{n}\geq -\beta\right\}}\mathbf{1}_{\left\{S_{n}\in B\right\}}]}{R^{-}(a+\beta)}.
\end{equation}

Similarly, define the weak ascending ladder epochs of random walk $(S_{n},n\geq0)$ as
$$\Gamma_{0}:=0,~ \Gamma_{k+1}:=\inf\left\{n>\Gamma_{k}:S_{n}\geq S_{\Gamma_{k}}\right\},~ k\geq0.$$
Let $R(x)$ be the renewal function associated with the weak ascending ladder height process $\left(S_{\Gamma_{n}},n\geq1\right)$, i.e.,
$$R(0):=1, ~R(x):=\sum_{n=1}^{\infty}\mu^{\infty}(S_{\Gamma_{n}}< x), ~x>0,$$
and denote by $\mathcal{R}(dx)$ the corresponding renewal measure. Define $R^{(\beta)}(x):=R(x+\beta)$ with the corresponding measure $\mathcal{R}^{(\beta)}(dx)$. By the renewal theorem again, there exists $c_{1},c_{2}>0$ such that for any non-negative measurable function $f$,
\begin{equation}\label{relation of renewal measure with L measure}
	c_{1}\int_{0}^{\infty}f(x-\beta)\,dx\leq \int_{-\beta}^{\infty}f(x)\mathcal{R}^{(\beta)}(dx)\leq c_{2}\int_{0}^{\infty}f(x-\beta)\,dx.
\end{equation}

We give the following lemma, which allows us to express the expectation of the series of the form $\sum_{n=1}^{\infty}\frac{G\left(\theta^n\xi,\zeta^{(\beta)}_{n}\right)}{U(\theta^n\xi,\zeta^{(\beta)}_{n}+\beta)}$ (where $G\left(\theta^n\xi,x\right)$ is a non-negative measurable function depending on the $n$-step shifted environment $\theta^n\xi$) under annealed probability $\mathbb{P}$ in the form of the integral with respect to $\mathcal{R}^{(\beta)}(dx)$.
\begin{lemma}\label{lem the integral of series of zeta_{n}}
	Let $\zeta^{(\beta)}_{n}$ be defined as $\left(\ref{Prob of conditioned rw}\right)$, for almost all $\xi$, $G\left(\xi,\cdot\right):[-\beta,\infty)\to[0,\infty)$ be a measurable function, $G\left(x\right):=\mathbf{E}\left[G\left(\xi,x\right)\right]<\infty$ for all $x\geq -\beta$, then\footnote{Here we have not included the starting point $\zeta^{(0)}_{0}$ in the summation term, with regard  the definition of renewal function $R(x)$.}
	\begin{equation}\label{the integral of series of zeta_{n}}
		\mathbb{E}\left[\sum_{n=1}^{\infty}\frac{G\left(\theta^n\xi,\zeta^{(\beta)}_{n}\right)}{U(\theta^n\xi,\zeta^{(\beta)}_{n}+\beta)}U\left(\xi,\beta\right)\right]=\int_{-\beta}^{\infty}G(x)\mathcal{R}^{(\beta)}(dx).
	\end{equation}
\end{lemma}

\begin{proof}
	By the duality principle for classical random walk, following the arguments of Section 2 and 6 of \cite{Big03}, for any non-negative measurable function $f$, we get
	\begin{equation}\label{intergral of R}
		\sum_{n=1}^{\infty}\mathbb{E}_{\mu^{\infty}}\left(f\left(S_{n}\right)\mathbf{1}_{\left\{\underline{S}_{n}\geq -\beta\right\}}\right)=\int_{-\beta}^{\infty}f(x)\mathcal{R}^{(\beta)}(dx).
	\end{equation}
	Then, by the definition of $\zeta^{(\beta)}_{n}$, we obtain
	\begin{equation}\nonumber
		\begin{aligned}
			\mathbb{E}\left[\sum_{n=1}^{\infty}\frac{G\left(\theta^n\xi,\zeta^{(\beta)}_{n}\right)}{U(\theta^n\xi,\zeta^{(\beta)}_{n}+\beta)}U\left(\xi,\beta\right)\right]=&\mathbf{E}\left[\sum_{n=1}^{\infty}\mathbb{E}_{\xi}\left(\frac{G\left(\theta^n\xi,\zeta^{(\beta)}_{n}\right)}{U(\theta^n\xi,\zeta^{(\beta)}_{n}+\beta)}\right)U\left(\xi,\beta\right)\right]\\
			=&\mathbf{E}\left[\sum_{n=1}^{\infty}\mathbb{E}_{\xi}\left(\frac{U(\theta^n\xi,S_n+\beta)\mathbf{1}_{\{\tau_{\beta}>n\}}G\left(\theta^n\xi,S_{n}\right)}{U(\theta^n\xi,S_{n}+\beta)U(\xi,\beta)}\right)U\left(\xi,\beta\right)\right]\\
			=&\mathbf{E}\left[\sum_{n=1}^{\infty}\mathbb{E}_{\xi}\left(\mathbf{1}_{\{\tau_{\beta}>n\}}G\left(\theta^n\xi,S_{n}\right)\right)\right]\\
			=&\mathbf{E}\left[\sum_{n=1}^{\infty}\int_{-\beta}^{\infty}G\left(\theta^n\xi,x\right)\mathbb{P}_{\xi}\left(S_{n}\in dx,\tau_{\beta}>n\right)\right].
		\end{aligned}
	\end{equation}
    Due to the i.i.d. random environment, for all fixed $x$, $G\left(\theta^n\xi,x\right)$ is a stationary and ergodic sequence (see, e.g. \cite{Kal02}, Lemmas 10.1 and 10.5), $\mathbf{E}\left[G\left(\theta^n\xi,x\right)\right]=G(x)$, and by the independence, we deduce
    \begin{equation}\nonumber
    	\begin{aligned}
    		\mathbf{E}\left[\sum_{n=1}^{\infty}\int_{-\beta}^{\infty}G\left(\theta^n\xi,x\right)\mathbb{P}_{\xi}\left(S_{n}\in dx,\tau_{\beta}>n\right)\right]=&\sum_{n=1}^{\infty}\int_{-\beta}^{\infty}\mathbf{E}\left[G\left(\theta^n\xi,x\right)\mathbb{P}_{\xi}\left(S_{n}\in dx,\tau_{\beta}>n\right)\right]\\
    		=&\int_{-\beta}^{\infty}G(x)\sum_{n=1}^{\infty}\mathbb{P}\left(S_{n}\in dx,\tau_{\beta}>n\right)\\
    		=&\int_{-\beta}^{\infty}G(x)\sum_{n=1}^{\infty}\mu^{\infty}\left(S_{n}\in dx,\underline{S}_{n}\geq -\beta\right)\\
    		=&\int_{-\beta}^{\infty}G(x)\mathcal{R}^{(\beta)}(dx),
    	\end{aligned}
    \end{equation}
    where the last equality follows from $\left(\ref{intergral of R}\right)$. Hence, this yields the lemma.
\end{proof}

\subsubsection{Quenched Tanaka's decomposition }

Tanaka's decomposition is an important tool for investigating the behavior of the random walk conditioned to stay non-negative, see \cite{Tan89,Big03,AGKV05} for example. In our context, with the 
preparations above, now we can specify a quenched version of Tanaka's decomposition for the RWRE conditioned to stay non-negative. And discuss the relationship between the two kind of  probability. For simplicity, we only consider $\beta=0$ and write $\zeta_{n}:=\zeta^{(0)}_{n}$. Let $\nu$ be the time of the first prospective minimal value of the process $(\zeta_{n},n\geqslant0)$, i.e.,
\begin{equation}\label{prospective minimal}
	\nu:=\inf\left\{m\geq1:\zeta_{m+n}\geq \zeta_{m}, ~\text{for all}~ n\geq0\right\}.
\end{equation}
Write $\zeta^{\nu}_{k}:=\zeta_{\nu+k}-\zeta_{\nu}$, $k\geq1$.
\begin{proposition}[Quenched Tanaka's decomposition]\label{Tanaka decomposition}
	For almost all $\xi$, we have
	
	$\left(1\right)$ $\zeta_{n}\to\infty$ $\mathbb{P}_{\xi}\text{-a.s.}$ as $n\to\infty$.
	
	$\left(2\right)$ $\left(\nu,\zeta_{1},\cdots,\zeta_{\nu}\right)$ and $\left(\zeta^{\nu}_{1},\zeta^{\nu}_{2},\cdots\right)$ are independent with respect to $\mathbb{P}_{\xi}$.
		
	$\left(3\right)$ $U(\xi,0)\mathbb{P}_{\xi}\left(\nu=k,\zeta_{\nu}\in dx\right)=U\left(\theta^k\xi,0\right)\mathbb{P}_{\xi}\left(S_{k}< S_{k-1},\cdots,S_{k}< S_{1},S_{k}\in dx\right)$ for all $k\geq1$.
\end{proposition}
	
\begin{proof}
	$\left(1\right)$ We claim that $U(y):=\mathbf{E}\left[U\left(\xi,y\right)\right]<\infty$ for any $y\geq0$. In fact, by the definition of $U\left(\xi,y\right)$, we have
	\begin{equation}\nonumber
		\begin{aligned}
			U(0)=\mathbf{E}\left[\mathbb{E}_{\xi}\left(-S_{\tau_{0}}\right)\right]=&\mathbf{E}\left[\sum_{n=1}^{\infty}\mathbb{E}_{\xi}\left(-S_{n},\tau_{0}=n\right)\right]\\
			=&\mathbf{E}\left[\sum_{n=1}^{\infty}\mathbb{E}_{\xi}\left(-S_{n},S_{1}\geq 0,\cdots,S_{n-1}\geq 0,S_{n}<0\right)\right]\\
			=&\sum_{n=1}^{\infty}\mathbb{E}_{\mu^{\infty}}\left(-S_{n},S_{1}\geq 0,\cdots,S_{n-1}\geq 0,S_{n}<0\right)\\
			=&\sum_{n=1}^{\infty}\mathbb{E}_{\mu^{\infty}}\left(-S_{n},\gamma_{1}=n\right)\\
			=&-\mathbb{E}_{\mu^{\infty}}\left(S_{\gamma_{1}}\right)\\
			<&\infty,
		\end{aligned}
	\end{equation}
	where the finiteness of $\mathbb{E}_{\mu^{\infty}}\left(S_{\gamma_{1}}\right)$ is valid by Theorem 1 in Chapter XVIII.5 of \cite{Fel71}. And by Proposition \ref{harmonic function} $\left(2\right)$, $U(y)$ satisfies the annealed harmonic property:
	\begin{equation}\label{annealed harmonic}
		\mathbb{E}\left[U\left(y+S_{1}\right)\mathbf{1}_{\left\{y+S_{1}\geq 0\right\}}\right]=U(y), ~~y\geq0.
	\end{equation}
	Since $\mathbb{P}\left(S_{1}>y_{0}\right)>0$ for some $y_{0}>0$. Then, by $\left(\ref{annealed harmonic}\right)$ with $y=0$, we have $U(y_{1})<\infty$ for some $y_{1}>y_{0}$. By again $\left(\ref{annealed harmonic}\right)$ with $y=y_{1}$, $U(y_{2})<\infty$ for some $y_{2}>y_{1}+y_{0}$. Repeating this argument, it follows that there exists a sequence $\left(y_{n},n\geq1\right)$ such that $U\left(y_{n}\right)<\infty$ for all $n$. By the monotonicity of $U(y)$, we conclude that $U(y)<\infty$ for all $y\geq0$.
	
	Since $U(\xi,y)$ is positive, non-decreasing in $y$, for any $y>0$, by the definition of $\zeta_{n}$, we have
	\begin{equation}\nonumber
		\begin{aligned}
			\mathbf{E}\left[U\left(\xi,0\right)\sum_{n=1}^{\infty}\mathbb{P}_{\xi}\left(\zeta_{n}<y\right)\right]=&\mathbf{E}\left[\sum_{n=1}^{\infty}\frac{\mathbb{E}_{\xi}\left(U(\theta^n\xi,S_n)\mathbf{1}_{\{\tau_{0}>n\}}\mathbf{1}_{\left\{S_{n}<y\right\}}\right)}{U(\xi,0)}U\left(\xi,0\right)\right]\\
			\leq&\sum_{n=1}^{\infty}\mathbf{E}\left[U(\theta^n\xi,y)\mathbb{E}_{\xi}\left(\mathbf{1}_{\{\tau_{0}>n\}}\mathbf{1}_{\left\{S_{n}<y\right\}}\right)\right].
		\end{aligned}
	\end{equation}
	Following the same argument as the proof of Lemma \ref{lem the integral of series of zeta_{n}}, we have
	\begin{equation}\nonumber
		\begin{aligned}
			\sum_{n=1}^{\infty}\mathbf{E}\left[U(\theta^n\xi,y)\mathbb{E}_{\xi}\left(\mathbf{1}_{\{\tau_{0}>n\}}\mathbf{1}_{\left\{S_{n}<y\right\}}\right)\right]=&\sum_{n=1}^{\infty}U(y)\mathbb{P}\left(S_{n}<y,\tau_{0}>n\right)\\
			=&U(y)\sum_{n=1}^{\infty}\mu^{\infty}\left(S_{n}<y,\underline{S}_{n}\geq0\right)\\
			=&U(y)\int_{0}^{y}\mathcal{R}^{(0)}(dx)\\
			<&\infty.
		\end{aligned}
	\end{equation}
	Thus, for almost all $\xi$, $\sum_{n=1}^{\infty}\mathbb{P}_{\xi}\left(\zeta_{n}<y\right)<\infty$. The Borel-Cantelli lemma yields that, for almost all $\xi$, $\zeta_{n}\to\infty$ $\mathbb{P}_{\xi}\text{-a.s.}$ as $n\to\infty$.
	
	$\left(2\right)$ Let
	$$H\left(\xi,x,z\right):=\frac{U\left(\xi,x-z\right)}{U\left(\xi,x\right)}, ~~x\geq z\geq0, ~~~~H\left(\xi,x,z\right):=0, ~~x<z,$$
	and $$\mathbb{P}^+_{\theta^{j-1}\xi}\left(y_{j-1};dy_{j}\right):=\mathbb{P}_{\xi}\left.\left(\zeta_{j}\in dy_{j}~\right| \zeta_{j-1}=y_{j-1}\right)=\frac{U(\theta^j\xi,y_{j})\mathbf{1}_{\left\{y_{j}\geq 0\right\}}}{U(\theta^{j-1}\xi,y_{j-1})}\mathbb{P}_{\theta^{j-1}\xi,y_{j-1}}\left(S_{1}\in dy_{j}\right), j\geq1.$$ Then, by Proposition \ref{harmonic function} $\left(2\right)$, $H\left(\xi,\cdot,z\right)$ is quenched harmonic with respect to the transition kernel $\mathbb{P}^+_{\xi}$ in the following sense:
	\begin{equation}\label{quenched harmonic 2}
		\int H\left(\theta\xi,y,z\right)\mathbb{P}^+_{\xi}\left(x;dy\right)=H\left(\xi,x,z\right), ~~x\geq z\geq0.
	\end{equation}
	Define $\hat{\tau}_{z}:=\inf\left\{n>0:\zeta_{n}< z\right\}$. $\hat{\tau}_{z}$ is a stopping time, it follows from $\left(\ref{quenched harmonic 2}\right)$ that the process $\left(H\left(\theta^{n\wedge\hat{\tau}_{z}}\xi,\zeta_{n\wedge\hat{\tau}_{z}},z\right),n\geq 0\right)$ is a martingale under $\mathbb{P}_{\xi}$. Thus, for all $n\geq 0$, we have
	$$\mathbb{E}_{\xi,x}\left[H\left(\theta^{n\wedge\hat{\tau}_{z}}\xi,\zeta_{n\wedge\hat{\tau}_{z}},z\right) \right]=H\left(\xi,x,z\right).$$
	Note that for almost all $\xi$, $0\leq H\left(\xi,x,z\right)\leq1$, $H\left(\xi,x,z\right):=0$ for $x<z$, and $H\left(\theta^n\xi,\zeta_{n},z\right)\to1$ as $n\to\infty$ by $\left(\ref{asymptotic behaviour of harmonic function 2}\right)$ and $\left(1\right)$. It follows that $$\lim_{n\to\infty}H\left(\theta^{n\wedge\hat{\tau}_{z}}\xi,\zeta_{n\wedge\hat{\tau}_{z}},z\right)=\lim_{n\to\infty}H\left(\theta^{n}\xi,\zeta_{n},z\right)\mathbf{1}_{\left\{\hat{\tau}_{z}>n\right\}}=\mathbf{1}_{\left\{\zeta_{n}\geq z~\text{for all}~ n>0\right\}}.$$
	By the dominated convergence theorem, we get
	$$\mathbb{P}_{\xi,x}\left(\zeta_{n}\geq z ~\text{for all}~ n>0\right)=\mathbb{E}_{\xi,x}\left[\lim_{n\to\infty}H\left(\theta^{n\wedge\hat{\tau}_{z}}\xi,\zeta_{n\wedge\hat{\tau}_{z}},z\right) \right]=H\left(\xi,x,z\right).$$
	This tell us that $H\left(\xi,x,z\right)$ is the probability that, starting at $x$, the process $\left(\zeta_{n},n\geq0\right)$ never hits $\left(-\infty,z\right)$.
	
	For any $x_{1},\cdots,x_{k},y_{1},\cdots,y_{m}\geq 0$ with $x_{0}=y_{0}=0$, we have
	\begin{equation}\nonumber
		\begin{aligned}
			&\prod_{j=1}^{m}\mathbb{P}^+_{\theta^{k+j-1}\xi}\left(y_{j-1}+x_{k};dy_{j}+x_{k}\right)H\left(\theta^{k+m}\xi,y_{m}+x_{k},x_{k}\right)\\
			=&\prod_{j=1}^{m}\mathbb{P}^+_{\theta^{k+j-1}\xi}\left(y_{j-1}+x_{k};dy_{j}+x_{k}\right)\frac{U\left(\theta^{k+m}\xi,y_{m}\right)}{U\left(\theta^{k+m}\xi,y_{m}+x_{k}\right)}\\
			=&\prod_{j=1}^{m}\mathbb{P}_{\theta^{k+j-1}\xi}\left(y_{j-1}+x_{k};dy_{j}+x_{k}\right)\frac{U\left(\theta^{k+j}\xi,y_{j}+x_{k}\right)}{U\left(\theta^{k+j-1}\xi,y_{j-1}+x_{k}\right)}\frac{U\left(\theta^{k+m}\xi,y_{m}\right)}{U\left(\theta^{k+m}\xi,y_{m}+x_{k}\right)}\\
			=&\prod_{j=1}^{m}\mathbb{P}_{\theta^{k+j-1}\xi}\left(y_{j-1};dy_{j}\right)\frac{U\left(\theta^{k+m}\xi,y_{m}\right)}{U\left(\theta^{k}\xi,x_{k}\right)}\\
			=&\prod_{j=1}^{m}\mathbb{P}_{\theta^{k+j-1}\xi}\left(y_{j-1};dy_{j}\right)\frac{U\left(\theta^{k+j}\xi,y_{j}\right)}{U\left(\theta^{k+j-1}\xi,y_{j-1}\right)}\frac{U\left(\theta^{k}\xi,0\right)}{U\left(\theta^{k}\xi,x_{k}\right)}\\
			=&\prod_{j=1}^{m}\mathbb{P}^+_{\theta^{k+j-1}\xi}\left(y_{j-1};dy_{j}\right)H\left(\theta^{k}\xi,x_{k},x_{k}\right).
		\end{aligned}
	\end{equation}
	As a result,
	\begin{equation}\nonumber
		\begin{aligned}
			&\mathbb{P}_{\xi}\left(\nu=k,\zeta_{1}\in dx_{1},\cdots,\zeta_{k}\in dx_{k},\zeta^{\nu}_{1}\in dy_{1},\cdots,\zeta^{\nu}_{m}\in dy_{m}\right)\\
			=&\mathbf{1}_{\left\{x_{1},\cdots,x_{k-1}>x_{k}\right\}}\mathbf{1}_{\left\{y_{1},\cdots,y_{m}\geq0\right\}}\prod_{i=1}^{k}\mathbb{P}^+_{\theta^i\xi}\left(x_{i-1};dx_{i}\right)\\
			&\times\prod_{j=1}^{m}\mathbb{P}^+_{\theta^{k+j-1}\xi}\left(y_{j-1}+x_{k};dy_{j}+x_{k}\right)H\left(\theta^{k+m}\xi,y_{m}+x_{k},x_{k}\right)\\
			=&\mathbf{1}_{\left\{x_{1},\cdots,x_{k-1}>x_{k}\right\}}\prod_{i=1}^{k}\mathbb{P}^+_{\theta^i\xi}\left(x_{i-1};dx_{i}\right)H\left(\theta^{k}\xi,x_{k},x_{k}\right)\prod_{j=1}^{m}\mathbb{P}^+_{\theta^{k+j-1}\xi}\left(y_{j-1};dy_{j}\right)\\
			=&\mathbb{P}_{\xi}\left(\nu=k,\zeta_{1}\in dx_{1},\cdots,\zeta_{k}\in dx_{k}\right)\mathbb{P}_{\xi}\left(\zeta^{\nu}_{1}\in dy_{1},\cdots,\zeta^{\nu}_{m}\in dy_{m}\right),
		\end{aligned}
	\end{equation}
	which proves that $\left(\nu,\zeta_{1},\cdots,\zeta_{\nu}\right)$ and $\left(\zeta^{\nu}_{1},\zeta^{\nu}_{2},\cdots\right)$ are independent with respect to $\mathbb{P}_{\xi}$.
	
	$\left(3\right)$ For all $k\geq1$ and $x\geq 0$, we have
	\begin{equation}\nonumber
		\begin{aligned}
			\mathbb{P}_{\xi}\left(\nu=k,\zeta_{\nu}\in dx\right)=&\mathbb{P}_{\xi}\left(\zeta_{k}< \zeta_{k-1},\cdots,\zeta_{k}< \zeta_{1},\zeta_{k}\in dx\right)H\left(\theta^{k}\xi,x,x\right)\\
			=&\mathbb{P}_{\xi}\left(S_{k}< S_{k-1},\cdots,S_{k}< S_{1},S_{k}\in dx\right)\frac{U\left(\theta^k\xi,x\right)}{U(\xi,0)}\frac{U\left(\theta^k\xi,0\right)}{U(\theta^k\xi,x)}\\
			=&\mathbb{P}_{\xi}\left(S_{k}< S_{k-1},\cdots,S_{k}< S_{1},S_{k}\in dx\right)\frac{U\left(\theta^k\xi,0\right)}{U(\xi,0)}
		\end{aligned}
	\end{equation}
\end{proof}

Define $\tilde{\mathbb{P}}(\cdot):=\int_{\Omega}\frac{U(\xi,0)}{U(0)}\mathbb{P}_{\xi}(\cdot)\, d\mathbf{P}$ and denote by $\tilde{\mathbb{E}}$ the corresponding expectation, we show that the excursion $\left(\zeta_{1},\cdots,\zeta_{\nu}\right)$ under annealed probability $\tilde{\mathbb{P}}$ is distributed as $\left(S_{\Gamma_1}-S_{\Gamma_1-1},\cdots,S_{\Gamma_1}\right)$ under $\mu^{\infty}$.
\begin{proposition}[Annealed excursion distribution]\label{Annealed excursion distribution}
	$\nu$ is the time of the first prospective minimal value of the process $(\zeta_{n},n\geqslant0)$ defined as $\left(\ref{prospective minimal}\right)$, then we have
	
	$\left(1\right)$ $\tilde{\mathbb{P}}\left(\zeta_{\nu}\in dx\right)=\mu^{\infty}\left(S_{\Gamma_1}\in dx\right)$.
	
	$\left(2\right)$ $\tilde{\mathbb{E}}\left[f\left(\nu,\zeta_{1},\cdots,\zeta_{\nu}\right)\right]=\mathbb{E}_{\mu^{\infty}}\left[f\left(\Gamma_1,S_{\Gamma_1}-S_{\Gamma_1-1},\cdots,S_{\Gamma_1}\right)\right]$ for any measurable function $f$.
\end{proposition}

\begin{proof}
	$\left(1\right)$ Proposition \ref{Tanaka decomposition} (3) yields
	\begin{equation}\nonumber
		\begin{aligned}
			\mathbf{E}\left[U(\xi,0)\mathbb{P}_{\xi}\left(\nu=k,\zeta_{\nu}\in dx\right)\right]=&\mathbf{E}\left[\mathbb{P}_{\xi}\left(S_{k}< S_{k-1},\cdots,S_{k}< S_{1},S_{k}\in dx\right)U\left(\theta^k\xi,0\right)\right]\\
			=&U\left(0\right)\mu^{\infty}\left(S_{k}-S_{k-1}<0,\cdots,S_{k}-S_{1}<0,S_{k}\in dx\right)\\
			=&U\left(0\right)\mu^{\infty}\left(S_{1}<0,\cdots,S_{k-1}<0,S_{k}\in dx\right)\\
			=&U\left(0\right)\mu^{\infty}\left(\Gamma_1=k,S_{\Gamma_1}\in dx\right).
		\end{aligned}
	\end{equation}
    Dividing by $U(0)$ and summing over $k$, we have $\tilde{\mathbb{P}}\left(\zeta_{\nu}\in dx\right)=\mu^{\infty}\left(S_{\Gamma_1}\in dx\right)$.
    
    $\left(2\right)$ Similar to the proof of Proposition \ref{Tanaka decomposition} (3), we get
    \begin{equation}\nonumber
    	\begin{aligned}
    		&\tilde{\mathbb{E}}\left[f\left(\nu,\zeta_{1},\cdots,\zeta_{\nu}\right)\right]\\
    		=&\mathbf{E}\left[\int f\left(k,x_{1},\cdots,x_{k}\right)\frac{U(\xi,0)}{U(0)}\mathbb{P}_{\xi}\left(\nu=k,\zeta_{1}\in dx_{1},\cdots,\zeta_{k}\in dx_{k}\right)\right]\\
    		=&\mathbf{E}\left[\int f\left(k,x_{1},\cdots,x_{k}\right)\frac{U\left(\theta^k\xi,0\right)}{U(0)}\mathbb{P}_{\xi}\left(S_{k}< S_{k-1},\cdots,S_{k}< S_{1}, S_{1}\in dx_{1},\cdots,S_{k}\in dx_{k}\right)\right]\\
    		=&\int f\left(k,x_{1},\cdots,x_{k}\right)\mu^{\infty}\left(S_{k}< S_{k-1},\cdots,S_{k}< S_{1}, S_{1}\in dx_{1},\cdots,S_{k}\in dx_{k}\right)\\
    		=&\int f\left(k,x_{1},\cdots,x_{k}\right)\mu^{\infty}\left(\tilde{S}_{1}<0,\cdots,\tilde{S}_{k-1}<0, \tilde{S}_{k}-\tilde{S}_{k-1}\in dx_{1},\cdots,\tilde{S}_{k}\in dx_{k}\right)\\
    		=&\mathbb{E}_{\mu^{\infty}}\left[f\left(\Gamma_1,S_{\Gamma_1}-S_{\Gamma_1-1},\cdots,S_{\Gamma_1}\right)\right],
    	\end{aligned}
    \end{equation}
    where $\tilde{S}_{j}:=S_{k}-S_{k-j},j\leq k$ and the last equality follows from the duality property for the random walk under $\mu^{\infty}$.
\end{proof}

\subsubsection{Application of Tanaka's decomposition }

The following proposition gives a necessary and sufficient condition for the a.s. divergence of some series associated with $\zeta^{(\beta)}_{n}$.
\begin{proposition}\label{prop 0-1 law}
	Let $\zeta^{(\beta)}_{n}$ be defined as $\left(\ref{Prob of conditioned rw}\right)$, $F:[-\beta,\infty)\to[0,\infty)$ be a non-increasing measurable function, for all $x\geq -\beta$, then
	\begin{equation}\nonumber
		\begin{aligned}
			\int_{-\beta}^{\infty}F(x)(x+\beta)\, dx=\infty ~~&\Longleftrightarrow~~\sum_{n=1}^{\infty}\frac{U\left(\xi,\beta\right)F\left(\zeta^{(\beta)}_{n}\right)\left(\zeta^{(\beta)}_{n}+\beta\right)}{U\left(\theta^{n}\xi,\zeta^{(\beta)}_{n}+\beta\right)}=\infty,~~\mathbb{P}\text{-a.s.}\\
			~~&\Longleftrightarrow~~\sum_{n=1}^{\infty}U\left(\xi,\beta\right)F\left(\zeta^{(\beta)}_{n}\right)=\infty,~~\mathbb{P}\text{-a.s.}.
		\end{aligned}
	\end{equation}
\end{proposition}

\begin{proof}
	For the second equivalence relation, note that for almost all $\xi$, $\zeta^{(\beta)}_{n}+\beta\to\infty$ $\mathbb{P}_{\xi}\text{-a.s.}$ as $n\to\infty$ (Proposition \ref{Tanaka decomposition} $\left(1\right)$) and by $\left(\ref{asymptotic behaviour of harmonic function 2}\right)$, we have $U\left(\theta^n\xi,\zeta^{(\beta)}_{n}+\beta\right)\sim\zeta^{(\beta)}_{n}+\beta$ as $n\to\infty$, hence,
	$$\sum_{n=1}^{\infty}\frac{U\left(\xi,\beta\right)F\left(\zeta^{(\beta)}_{n}\right)\left(\zeta^{(\beta)}_{n}+\beta\right)}{U\left(\theta^{n}\xi,\zeta^{(\beta)}_{n}+\beta\right)}=\infty~~\Longleftrightarrow~~\sum_{n=1}^{\infty}U\left(\xi,\beta\right)F\left(\zeta^{(\beta)}_{n}\right)=\infty, ~~\mathbb{P}\text{-a.s.}.$$
	By Lemma \ref{lem the integral of series of zeta_{n}}, we have
	$$\mathbb{E}\left[\sum_{n=1}^{\infty}\frac{U\left(\xi,\beta\right)F\left(\zeta^{(\beta)}_{n}\right)\left(\zeta^{(\beta)}_{n}+\beta\right)}{U\left(\theta^{n}\xi,\zeta^{(\beta)}_{n}+\beta\right)}\right]=\int_{-\beta}^{\infty}F(x)(x+\beta)\mathcal{R}^{(\beta)}(dx).$$
	It follows from (\ref{relation of renewal measure with L measure}) that
	$$\sum_{n=1}^{\infty}\frac{U\left(\xi,\beta\right)F\left(\zeta^{(\beta)}_{n}\right)\left(\zeta^{(\beta)}_{n}+\beta\right)}{U\left(\theta^{n}\xi,\zeta^{(\beta)}_{n}+\beta\right)}=\infty,~~\mathbb{P}\text{-a.s.} ~~\Longrightarrow~~\int_{-\beta}^{\infty}F(x)(x+\beta)\, dx=\infty.$$
	We only need to prove that
	$$\int_{-\beta}^{\infty}F(x)(x+\beta)\, dx=\infty ~~\Longrightarrow~~\sum_{n=1}^{\infty}U\left(\xi,\beta\right)F\left(\zeta^{(\beta)}_{n}\right)=\infty,~~\mathbb{P}\text{-a.s.}.$$
	
	For simplicity, we only consider $\beta=0$ and write $\zeta_{n}:=\zeta^{(0)}_{n}$, since $\beta>0$ is similar with this case. Note that
	$$\sum_{n=1}^{\infty}U\left(\xi,0\right)F\left(\zeta_{n}\right)=\infty,~~\mathbb{P}\text{-a.s.}~~\Longleftrightarrow~~\sum_{n=1}^{\infty}F\left(\zeta_{n}\right)=\infty,~~\tilde{\mathbb{P}}\text{-a.s.}.$$
	We are left to show that
	\begin{equation}\label{0-1 law}
		\int_{0}^{\infty}F(x)x\, dx=\infty 
		~~\Longrightarrow~~\sum_{n=1}^{\infty}F\left(\zeta_{n}\right)=\infty,~~\tilde{\mathbb{P}}\text{-a.s.}.
	\end{equation}
	To prove (\ref{0-1 law}), it suffices to check that
	$$\tilde{\mathbb{P}}\left(\sum_{n=1}^{\infty}F\left(\zeta_n\right)=\infty\right)<1 ~~\Longrightarrow~~\int_0^{\infty} F(x)x \,dx<\infty.$$
	We assume that $\tilde{\mathbb{P}}\left(\sum_{n=1}^{\infty}F\left(\zeta_n\right)=\infty\right)<1$, that is $\tilde{\mathbb{P}}\left(\sum_{n=1}^{\infty}F\left(\zeta_n\right)<\infty\right)>0$.
	
	We first use Tanaka's decomposition (Proposition \ref{Tanaka decomposition}, \ref{Annealed excursion distribution}) to reconstruct the process $\left(\zeta_{n},n\geq0\right)$. Recall that $$\nu:=\inf\left\{m\geq1:\zeta_{m+n}\geq \zeta_{m}, ~\text{for all}~ n\geq0\right\}.$$
	We have an excursion $\left(\zeta_j, 0\leq j \leq \nu\right)$, which is denoted by $\omega=(\omega(j), 0\leq j \leq \nu)$. Let $\left\{\omega_k=\left(\omega_k(j), 0 \leq j \leq \nu_k\right), k \geq 1\right\}$ be a sequence of independent copies of $\omega$ under $\tilde{\mathbb{P}}$. Define
	$$V_0:=0,~~V_k:=\nu_1+\cdots+\nu_k,~\text{for all}~ k \geq 1.$$
	The process
	$$\zeta_0=0,~~\zeta_n=\sum_{i=1}^{k}\omega_i\left(\nu_i\right)+\omega_{k+1}\left(n-V_{k}\right), ~\text{for}~ V_k<n \leq V_{k+1},$$ is what we need. Then,
	\begin{equation}\nonumber
		\begin{aligned}
			\sum_{n=1}^{\infty}F\left(\zeta_{n}\right)&=\sum_{k=1}^{\infty}\sum_{n=V_{k-1}+1}^{V_k}F\left(\sum_{i=1}^{k-1}\omega_i\left(\nu_i\right)+\omega_{k}\left(n-V_{k-1}\right)\right)\\
			&=\sum_{k=1}^{\infty}\sum_{j=1}^{\nu_k}F\left(\sum_{i=1}^{k}\omega_i\left(\nu_i\right)-\left(\omega_{k}\left(\nu_k\right)-\omega_{k}\left(j\right)\right)\right).
		\end{aligned}
	\end{equation}
	Hence, by hypothesis,
	$$\tilde{\mathbb{P}}\left(\sum_{n=1}^{\infty}F\left(\zeta_{n}\right)<\infty\right)=\tilde{\mathbb{P}}\left(\sum_{k=1}^{\infty}\sum_{j=1}^{\nu_k}F\left(\sum_{i=1}^{k}\omega_i\left(\nu_i\right)-\left(\omega_{k}\left(\nu_k\right)-\omega_{k}\left(j\right)\right)\right)<\infty\right)>0.$$
	
	On the other hand, by Proposition \ref{Annealed excursion distribution} (1), it follows from the strong law of large number that
	$$\lim _{k \rightarrow \infty} \frac{\sum_{i=1}^{k}\omega_i\left(\nu_i\right)}{k}=C,~~\tilde{\mathbb{P}}\text{-a.s.}.$$
	where $C:=\mathbb{E}_{\mu^\infty}\left(S_{\Gamma_1}\right)$, the finiteness of $C$ is due to \cite{Fel71} (Chapter XVIII, Section 5, Theorem 1).	Let $\epsilon>0$ and $A:=(C+\epsilon)\vee 1$, thus, for all sufficiently large $k$, $\sum_{i=1}^{k}\omega_i\left(\nu_i\right)\leq Ak$. Since $F$ is non-increasing, we obtain
	\begin{equation}\nonumber
		\begin{aligned}
			&\tilde{\mathbb{P}}\left(\sum_{k=1}^{\infty}\sum_{j=1}^{\nu_k}F\left(Ak-\left(\omega_{k}\left(\nu_k\right)-\omega_{k}\left(j\right)\right)\right)<\infty\right)\\
			\geq~&\tilde{\mathbb{P}}\left(\sum_{k=1}^{\infty}\sum_{j=1}^{\nu_k}F\left(\sum_{i=1}^{k}\omega_i\left(\nu_i\right)-\left(\omega_{k}\left(\nu_k\right)-\omega_{k}\left(j\right)\right)\right)<\infty\right)>0.
		\end{aligned}
	\end{equation}
	Let
	$$\chi_k(\nu,\omega,F):=\sum_{j=1}^{\nu_k} F\left(Ak-\left(\omega_k(\nu_{k})-\omega_k(j)\right)\right),$$
	hence, $\tilde{\mathbb{P}}\left(\sum_{k=1}^{\infty}\chi_k(\nu,\omega,F)<\infty\right)>0$. Note that the independence of the sequence $\left\{\omega_k, k \geq 1\right\}$ yields the independence of the sequence $\left\{\chi_k\left(\nu,\omega,F\right), k \geq 1\right\}$. By Kolmogorov's 0-1 law, it follows that
	\begin{equation}\label{a.s. finite}
		\tilde{\mathbb{P}}\left(\sum_{k=1}^{\infty}\chi_k\left(\nu,\omega,F\right)<\infty\right)=1.
	\end{equation}
	
	From now on, we are proceeding in the same way as \cite{Che15}. Let $E_M:=\left\{\sum_{k=1}^{\infty}\chi_k\left(\nu,\omega,F\right)<M\right\}$ for any $M>0$. Either $\mathbb{P}\left(E_{M_0}\right)=1$ for some $M_0<\infty$, or $\tilde{\mathbb{P}}\left(E_M\right)<1$ for all $M \in(0, \infty)$. For the first case, that is there exists some $M_0<\infty$ such that $\mathbb{P}\left(E_{M_0}\right)=1$, then
	$$
	\begin{aligned}
		M_0 &\geq \tilde{\mathbb{E}}\left(\sum_{k=1}^{\infty}\chi_k\left(\nu,\omega,F\right)\right)\\
		&=\tilde{\mathbb{E}}\left(\sum_{k=1}^{\infty}\sum_{j=1}^{\nu_k}F\left(Ak-\left(\omega_k(\nu_{k})-\omega_k(j)\right)\right)\right)\\
		&=\sum_{k=1}^{\infty}\tilde{\mathbb{E}}\left(\sum_{j=1}^{\nu} F\left(Ak-\left(\omega(\nu)-\omega(j)\right)\right)\right)\\
		&=\sum_{k=1}^{\infty}\mu^{\infty}\left(\sum_{j=0}^{\Gamma_1-1} F\left(Ak-S_{j}\right)\right),
	\end{aligned}
	$$
	where the last equality follows from Proposition \ref{Annealed excursion distribution} (2). Similar to $\left(\ref{intergral of R}\right)$, we have
	$$
	\begin{aligned}
		\mathbb{E}_{\mu^{\infty}}\left(\sum_{j=0}^{\Gamma_1-1} F\left(Ak-S_{j}\right)\right)=&\mathbb{E}_{\mu^{\infty}}\left(\sum_{n=0}^{\infty}F\left(Ak-S_{n}\right)\mathbf{1}_{\left\{n<\Gamma_1\right\}}\right)\\
		=&\mathbb{E}_{\mu^{\infty}}\left(F\left(Ak-S_{n}\right)\mathbf{1}_{\left\{\max_{j\leq n}S_{j}<0\right\}}\right)\\
		=&\int_0^{\infty}F(Ak+x)\mathcal{R}^{-}(dx),
	\end{aligned}
	$$
	where $\mathcal{R}^{-}(dx)$ is the renewal measure of $R^{-}(x)$, i.e., the renewal measure associated with the strict descending ladder height process. Thus, $\sum_{k=1}^{\infty}\int_0^{\infty}F(Ak+x)\mathcal{R}^{-}(dx)<\infty$,
	which implies that
	$$\int_0^{\infty}F(x)x \,dx<\infty.$$
	For the second case, $\tilde{\mathbb{P}}\left(E_M\right)<1$ for all $M \in(0, \infty)$, so $\lim _{M \to \infty} \tilde{\mathbb{P}}\left(E_M\right)=1$ by (\ref{a.s. finite}). Let
	$$\Lambda_{k,l}(\nu,\omega):=\sum_{j=1}^{\nu_k}\mathbf{1}_{\left\{A(l-1)\leq-\left(\omega_k(\nu_{k})-\omega_k(j)\right)<Al\right\}}, ~~\text{for all}~~k\geq1, ~~l\geq1.$$
	Note that, for any $k \geq 1$,
	$$
	\begin{aligned}
		\chi_k\left(\nu,\omega,F\right)&=\sum_{j=1}^{\nu_k} F\left(Ak-\left(\omega_k(\nu_{k})-\omega_k(j)\right)\right) \sum_{l=1}^{\infty} \mathbf{1}_{\left\{A(l-1) \leq-\left(\omega_k(\nu_{k})-\omega_k(j)\right)<Al\right\}}\\
		&=\sum_{l=1}^{\infty} \sum_{j=1}^{\nu_k} F\left(Ak-\left(\omega_k(\nu_{k})-\omega_k(j)\right)\right) \mathbf{1}_{\left\{A(l-1) \leq-\left(\omega_k(\nu_{k})-\omega_k(j)\right)<Al\right\}}\\
		&\geq \sum_{l=1}^{\infty} F(Ak+Al) \Lambda_{k,l}(\nu,\omega),
	\end{aligned}
	$$
	where the last inequality holds because $F$ is non-increasing. Thus, we have
	$$
	\begin{aligned}
		\sum_{k=1}^{\infty} \chi_k\left(\nu,\omega,F\right) &\geq \sum_{k=1}^{\infty}\sum_{l=1}^{\infty}F(Ak+Al) \Lambda_{k,l}(\nu,\omega)\\
		&=\sum_{m=1}^{\infty} F(Am+A) \sum_{k=1}^{m} \Lambda_{k,m+1-k}(\nu,\omega)\\
		&=\sum_{m=1}^{\infty} F(Am+A) m Y_m,
	\end{aligned}
	$$
	where $Y_m:=\sum_{k=1}^m \Lambda_{k,m+1-k}(\nu,\omega) / m$ for all $m \geq 1$. Note that, $\left(\Lambda_{k,\cdot}(\nu,\omega), k\geq1\right)$ are i.i.d. under $\tilde{\mathbb{P}}$, and for all $l\geq1$, $\Lambda_{1,l}(\nu,\omega)$ has the same law as $\sum_{j=0}^{\Gamma_{1}-1}\mathbf{1}_{\left\{A(l-1)\leq -S_j< Al\right\}}$ under $\mu^{\infty}$. Following the same first and second moments argument for $Y_m$ as \cite{Che15}, we obtain that there exists a sufficiently large number $M>0$ such that, for any $m \geq 1$,
	$$
	C_2 \geq \tilde{\mathbb{E}}\left(Y_m \mathbf{1}_{E_M}\right) \geq C_1>0,
	$$
	where $C_1, C_2$ are positive constants. Therefore, we have
	$$
	\begin{aligned}
		M & \geq \tilde{\mathbb{E}}\left(\sum_{k=1}^{\infty} \chi_k\left(\nu,\omega,F\right) \mathbf{1}_{E_M}\right)\\
		& \geq \tilde{\mathbb{E}}\left(\sum_{m=1}^{\infty} F(Am+A) m Y_m \mathbf{1}_{E_M}\right)\\
		& = \sum_{m=1}^{\infty} F(Am+A) m \tilde{\mathbb{E}}\left(Y_m \mathbf{1}_{E_M}\right)\\
		& \geq \sum_{m=1}^{\infty} F(Am+A) m C_1.
	\end{aligned}
	$$
	This yields
	$$\sum_{m=1}^{\infty} F(A m+A) m \leq \frac{M}{C_1}<\infty,$$
	which implies that $\int_0^{\infty} F(y) y \mathrm{~d} y<\infty$, and completes the proof of (\ref{0-1 law}).
	
\end{proof}


\section{Change of measure and spinal decomposition}\label{section 3}

In this section, we use the quenched harmonic function of the associated random walk to introduce the truncated derivative martingale. We then show that the limit of the derivative martingale exists and the non-triviality of the limit is equivalent to the mean convergence of the truncated derivative martingales. Finally, we give a proof of the spinal decomposition for the time-inhomogeneous branching random walk. The main idea is similar as that of the constant environment situation, the proof is postpone to the appendix.

\subsection{Truncated derivative martingale and change of probabilities}

To study the limit of the derivative martingale, we introduce the non-negative process with a barrier.

Let $\beta\geq0$, $V\left(\varnothing\right)=a\geq0$, we define
$$D^{(\beta)}_{n}:=\sum_{|u|=n}U\left(\theta^{n}\xi,V(u)+\beta\right)e^{-V(u)}\mathbf{1}_{\left\{\min_{0\leq k\leq n}V(u_{k})\geq-\beta\right\}}, ~~n\geq1,$$
and $D^{(\beta)}_{0}:=U\left(\xi,a+\beta\right)e^{-a}$.

\begin{lemma}[Truncated martingale]\label{truncated}
	For any $\beta\geq0$ and $a\geq0$, the process $\left(D^{(\beta)}_{n},n\geq0\right)$ is a non-negative martingale with respect to the filtration $\left(\mathscr{F}_{n},n\geq0\right)$ under both laws $\mathbb{P}_{\xi,a}$ and $\mathbb{P}_{a}$. Therefore, for almost all $\xi$, $D^{(\beta)}_{n}$ converges $\mathbb{P}_{\xi,a}\text{-a.s.}$ to a non-negative finite limit, which we denote by $D^{(\beta)}_{\infty}$.
\end{lemma}

The proof see the appendix.

Thanks to the following lemma, the truncated martingales $\left(D^{(\beta)}_{n},n\geq0\right)$ do approach the derivative martingale $\left(D_{n},n\geq0\right)$ and we can study the non-triviality of the limit of the derivative martingale by the mean convergence of the truncated derivative martingales.

\begin{lemma}\label{connection}
	$\left(1\right)$ Assume that $\left(\ref{ass1}\right)$, $\left(\ref{boundary condition}\right)$ and $\left(\ref{2 more moment}\right)$ hold, then $\lim_{n\to\infty}D_{n}=D_{\infty}\geq0, ~\mathbb{P}\text{-a.s.}$.
		
	$\left(2\right)$ If there exists $\beta\geq0$ such that for almost all $\xi$, $D^{(\beta)}_{n}$ converges in $L^{1}\left(\mathbb{P}_{\xi}\right)$, then $D_{\infty}$ is non-degenerate for almost all $\xi$, i.e. $\mathbb{P}_{\xi}\left(D_{\infty}>0\right)>0, ~\mathbf{P}\text{-a.s.}$.
	
	$\left(3\right)$ If for almost all $\xi$, $D^{(\beta)}_{\infty}=0, ~\mathbb{P}_{\xi}\text{-a.s.}$ for all $\beta\geq0$, then $D_{\infty}$ is degenerate for almost all $\xi$, i.e. $\mathbb{P}_{\xi}\left(D_{\infty}=0\right)=1, ~\mathbf{P}\text{-a.s.}$.
\end{lemma}

The proof see the appendix.

\begin{remark}
	In proving our main Theorem, indeed, we also show that the following two statements are equivalent:
	
	$\left(i\right)$ There exists $\beta\geq0$ such that $D^{(\beta)}_{n}$ is $L^{1}\left(\mathbb{P}_{\xi}\right)$ convergence for almost all $\xi$.
	
	$\left(ii\right)$ For any $\beta\geq0$, $D^{(\beta)}_{n}$ is $L^{1}\left(\mathbb{P}_{\xi}\right)$ convergence for almost all $\xi$.
\end{remark}

Since $\left(D^{(\beta)}_{n},n\geq0\right)$ is a non-negative martingale with $\mathbb{E}_{\xi,a}\left(D^{(\beta)}_{n}\right)=U\left(\xi,a+\beta\right)e^{-a}$, it follows from Kolmogorov’s extension theorem that there exists a unique probability measure $\mathbb{Q}^{(\beta)}_{\xi,a}$ on $\mathscr{F}_{\infty}:=\vee_{n\geq0}\mathscr{F}_{n}$ such that for all $n\geq1$,
\begin{equation}\label{definition of quenched Q}
	\frac{d\mathbb{Q}^{(\beta)}_{\xi,a}}{d\mathbb{P}_{\xi,a}}\bigg|_{\mathscr{F}_{n}}:=\frac{D^{(\beta)}_{n}}{U\left(\xi,a+\beta\right)e^{-a}}.
\end{equation}
An intuitive description of the new probability is presented in the next subsection.

\subsection{Spinal decomposition of the time-inhomogeneous branching random walk}

This subsection is devoted to the proof of a time-inhomogeneous version of the spinal decomposition of the branching random walk. The spinal decomposition has been introduced by Lyons, Pemantle and Peres to study Galton-Watson processes in \cite{LPP95}. This result was then adapted by Lyons \cite{Lyo97} to studying the additive martingale for the branching random walk. Later, Biggins and Kyprianou \cite{BK04} extended this approach to treat general martingales based on additive functional of multitype branching.

The spinal decomposition gives an alternative description of the law of a branching random walk biased by a non-negative martingale as a branching random walk with a special individual called the ``spine''. Now, we introduce the time-inhomogeneous branching random walk with a spine $w^{(\beta)}=\left(w^{(\beta)}_{n},n\geq0\right)$ as follows. The process starts with one particle $w^{(\beta)}_{0}$ at position $V\left(w^{(\beta)}_{0}\right)=a$. It dies at time $1$ and gives birth to children distributed as $\hat{L}^{(\beta)}_{1,a}$ whose distribution is the law of $L_{1}$ under $\mathbb{Q}^{(\beta)}_{\xi,a}$. The particle $u$ is chosen as the spine element $w^{(\beta)}_{1}$ among the children of $w^{(\beta)}_{0}$ with probability proportional to $U\left(\theta\xi,V(u)\right)e^{-V(u)}\mathbf{1}_{\left\{V(u)\geq-\beta\right\}}$, while all other children are normal particles. For any $n\geq1$, each particle alive at generation $n$ dies at time $n+1$ and gives birth independently to children. The children of normal particle $z$ is distributed as $L_{n+1,V(z)}$ (i.e. $L_{n+1}$ under $\mathbb{P}_{\xi,V(z)}$), while the spine element $w^{(\beta)}_{n}$ produces according to the point process $\hat{L}^{(\beta)}_{n+1,V(w^{(\beta)}_{n})}$ which distributed as $L_{n+1}$ under $\mathbb{Q}^{(\beta)}_{\xi,V(w^{(\beta)}_{n})}$, and the particle $v$ is chosen as the spine element $w^{(\beta)}_{n+1}$ among the children of $w^{(\beta)}_{n}$ with proportional to $U\left(\theta^{n+1}\xi,V(v)+\beta\right)e^{-V(v)}\mathbf{1}_{\left\{\min_{0\leq k\leq n+1}V(v_{k})\geq-\beta\right\}}$, all other children are normal particles. The process goes on as described as above. We still denote by $\mathbb{T}$ the genealogical tree. Let us denote by $\hat{\mathbb{P}}^{(\beta)}_{\xi,a}$ the law of the new process, it is a probability on the product space between the space of all marked trees and the space of all infinite spine.

The following spinal decomposition consists of an alternative construction of the law $\mathbb{Q}^{(\beta)}_{\xi,a}$ as the projection of the law $\hat{\mathbb{P}}^{(\beta)}_{\xi,a}$ on the space of all marked trees. By an abuse of notation, the marginal law of $\hat{\mathbb{P}}^{(\beta)}_{\xi,a}$ on the space of marked trees is also denoted by $\hat{\mathbb{P}}^{(\beta)}_{\xi,a}$. This alternative construction allows us to study the mean convergence of the corresponding martingale in Section \ref{section 4}.

\begin{proposition}[Spinal decomposition]\label{spinal decomposition}
	The branching random walk umder $\mathbb{Q}^{(\beta)}_{\xi,a}$ is distributed as $\hat{\mathbb{P}}^{(\beta)}_{\xi,a}$.
\end{proposition}

The proof see the appendix.

\

Thanks to Proposition \ref{spinal decomposition}, we will identify the branching random walk under $\mathbb{Q}^{(\beta)}_{\xi,a}$ with $\hat{\mathbb{P}}^{(\beta)}_{\xi,a}$ in the following.

\begin{proposition}[Law of the spine]\label{law of the spine}
	Let the spine $w^{(\beta)}=\left(w^{(\beta)}_{n}\right)$ and probability $\mathbb{Q}^{(\beta)}_{\xi,a}$ be defined as above, we have
	
	$\left(1\right)$ For any $n$ and any vertex $v\in\mathbb{T}$ with $|v|=n$, we have
	$$\mathbb{Q}^{(\beta)}_{\xi,a}\left(w^{(\beta)}_{n}=v\mid \mathscr{F}_{n}\right)=\frac{U\left(\theta^{n}\xi,V(v)+\beta\right)e^{-V(v)}\mathbf{1}_{\left\{\min_{0\leq k\leq n}V(v_{k})\geq-\beta\right\}}}{D^{(\beta)}_{n}}.$$
	
	$\left(2\right)$ The process $\left(V(w^{(\beta)}_{n}),n\geq0\right)$ under $\mathbb{Q}^{(\beta)}_{\xi,a}$ is distributed as the random walk $\left(S_{n},n\geq0\right)$ under $\mathbb{P}_{\xi,a}$ conditioned to stay in $[-\beta,\infty)$. Or equivalently, for all $n$ and any measurable function $f:\mathbb{R}^{n+1}\to\mathbb{R}_{+}$, we have $\mathbf{P}\text{-a.s.}$
	\begin{equation}\label{law of the spine 1}
		\mathbb{E}_{\mathbb{Q}^{(\beta)}_{\xi,a}}\left[f\left(V(w^{(\beta)}_{0}), \cdots, V(w^{(\beta)}_{n})\right)\right]=\mathbb{E}_{\xi, a}\left[f\left(S_{0}, \cdots, S_{n}\right)\frac{U(\theta^{n}\xi,S_{n}+\beta)}{U(\xi,a+\beta)}\mathbf{1}_{\left\{\min_{0\leq k\leq n}S_{k}\geq-\beta\right\}}\right].
	\end{equation}
\end{proposition}

The proof see the appendix.

Note that $\mathbb{Q}^{(\beta)}_{\xi,a}\left(D^{(\beta)}_{n}>0\right)=\mathbb{E}_{\xi,a}\left[\frac{D^{(\beta)}_{n}}{U\left(\xi,a+\beta\right)e^{-a}}\right]=1$, the right-hand side of identity in Proposition \ref{law of the spine} (1) makes sense $\mathbb{Q}^{(\beta)}_{\xi,a}\text{-a.s.}$. For (2), by (\ref{Prob of conditioned rw}) and $\left(\ref{law of the spine 1}\right)$, we have the identity:
\begin{equation}\label{law of the spine 2}
	\mathbb{E}_{\mathbb{Q}^{(\beta)}_{\xi,a}}\left[f\left(V(w^{(\beta)}_{0}), \cdots, V(w^{(\beta)}_{n})\right)\right]=\mathbb{E}_{\xi, a}\left[f\left(\zeta^{(\beta)}_{0}, \cdots, \zeta^{(\beta)}_{n}\right)\right].
\end{equation}


\section{Proof of Theorem \ref{main result}}\label{section 4}

Lemma \ref{connection} $\left(1\right)$ yields Theorem \ref{main result} $\left(1\right)$. In this section, under the assumptions $\left(\ref{ass1}\right)$, $\left(\ref{boundary condition}\right)$ and $\left(\ref{2 more moment}\right)$, we prove that (\ref{iff}) is a necessary and sufficient condition for $L^{1}\left(\mathbb{P}_{\xi}\right)$ convergence of the truncated derivative martingale $\left(D^{(\beta)}_{n},n\geq0\right)$ for almost all $\xi$. Then by Lemma \ref{connection} $\left(2\right)$ and $\left(3\right)$, this condition is equivalent to the non-degeneracy of the limit $D_{\infty}$ of the derivative martingale, which proves Theorem \ref{main result} $\left(2\right)$.

Following the general treatment of Biggins and Kyprianou \cite{BK04} for multitype branching processes, which they use to obtain the mean convergence of the martingales produced by the mean-harmonic function, we can prove the following result (Proposition \ref{BK th 2.1}) considering the mean convergence of the truncated derivative martingale.

Recalling that $\overleftarrow{u}$ is the parent of $u$, for any $u\in\mathbb{T} \backslash \left\{\varnothing\right\}$, we define its relative position by
$$\Delta V(u):=V(u)-V(\overleftarrow{u}).$$
Under $\mathbb{P}_{\xi,\zeta^{(\beta)}_{n}}$, we define
\begin{equation}\label{tilde{X}}
	\begin{aligned}
		\tilde{X}:=&\frac{\sum_{|u|=1}U\left(\theta^{n+1}\xi,\zeta^{(\beta)}_{n}+\Delta V(u)+\beta\right)e^{-\zeta^{(\beta)}_{n}-\Delta V(u)}\mathbf{1}_{\left\{\zeta^{(\beta)}_{n}+\Delta V(u)\geq-\beta\right\}}}{U\left(\theta^{n}\xi,\zeta^{(\beta)}_{n}+\beta\right)e^{-\zeta^{(\beta)}_{n}}}\\
		=&\frac{\sum_{|u|=1}U\left(\theta^{n+1}\xi,\zeta^{(\beta)}_{n}+\Delta V(u)+\beta\right)e^{-\Delta V(u)}\mathbf{1}_{\left\{\zeta^{(\beta)}_{n}+\Delta V(u)\geq-\beta\right\}}}{U\left(\theta^{n}\xi,\zeta^{(\beta)}_{n}+\beta\right)},
	\end{aligned}
\end{equation}
where $\left(\Delta V(u),|u|=1\right)$ is independent of $\zeta^{(\beta)}_{n}$ under $\mathbb{P}_{\xi}$.

\begin{proposition}\label{BK th 2.1}
	Let $\zeta^{(\beta)}_{n}$ be defined as $\left(\ref{Prob of conditioned rw}\right)$, for all $\beta\geq0$, we have
	
	$\left(1\right)$ If for almost all $\xi$,
	\begin{equation}\nonumber
		\sum_{n=1}^{\infty}\mathbb{E}_{\xi,\zeta^{(\beta)}_{n}}\left[\tilde{X}\left(\left(U(\theta^{n}\xi,\zeta^{(\beta)}_{n}+\beta)e^{-\zeta^{(\beta)}_{n}}\tilde{X}\right) \wedge 1\right)\right]U(\xi,\beta)<\infty, ~~\mathbb{P}_{\xi}\text{-a.s.},
	\end{equation}
	then $\mathbb{E}_{\xi}\left[D^{(\beta)}_{\infty}\right]=U(\xi,\beta), ~\mathbf{P}\text{-a.s.}$.
	
	$\left(2\right)$ If for any $c\geq1$,
	\begin{equation}\nonumber
		\sum_{n=1}^{\infty}\mathbb{E}_{\zeta^{(\beta)}_{n}}\left[\tilde{X}\mathbf{1}_{\left\{U(\theta^{n}\xi,\zeta^{(\beta)}_{n}+\beta)e^{-\zeta^{(\beta)}_{n}}\tilde{X}\geq c\right\}}\right]U(\xi,\beta)=\infty, ~~\mathbb{P}\text{-a.s.},
	\end{equation}
	then $\mathbb{E}\left[D^{(\beta)}_{\infty}\right]=0$.
\end{proposition}

\begin{proof}
	$\left(1\right)$ Thanks to the harmonic function and the spinal decomposition that we have established in the former sections, this result follows by the same argument as the proof of Theorem 2.1 in \cite{BK04}.
	
	$\left(2\right)$ For the degeneracy case, it is slightly different from the proof of $\left(1\right)$ since we use the annealed probability instead of quenched probability in the expression. We only point out the changes that should be made in the proof because the main idea is the sane as \cite{BK04}. Let
	$$\mathbb{Q}^{(\beta)}_{a}(\cdot):=\mathbf{E}\left[\mathbb{Q}^{(\beta)}_{\xi,a}(\cdot)\right],$$
	then we have
	\begin{equation}\nonumber
		\frac{d\mathbb{Q}^{(\beta)}_{a}}{d\mathbb{P}_{a}}\bigg|_{\mathscr{F}_{n}}=\frac{D^{(\beta)}_{n}}{U\left(\xi,a+\beta\right)e^{-a}}.
	\end{equation}
    It follows from the Corollary 1 of Athreya \cite{Ath00} that
    $$\mathbb{E}_{\mathbb{Q}}\left[D^{(\beta)}_{\infty}\right]=0\Longleftrightarrow \mathbb{Q}\left(D^{(\beta)}_{\infty}=\infty\right)=1.$$
    Note that, by the spinal decomposition,
    $$\mathbb{Q}\left(D^{(\beta)}_{\infty}=\infty\right)\geq \mathbb{Q}\left(\limsup_{n\to\infty}U\left(\theta^{n}\xi,V\left(w^{(\beta)}_{n}\right)+\beta\right)e^{-V\left(w^{(\beta)}_{n}\right)}\tilde{X}\left(V\left(w^{(\beta)}_{n}\right)\right)=\infty\right).$$
    Thus, it suffices to prove that, for any $c\geq1$,
    $$\limsup_{n\to\infty}U\left(\theta^{n}\xi,V\left(w^{(\beta)}_{n}\right)+\beta\right)e^{-V\left(w^{(\beta)}_{n}\right)}\tilde{X}\left(V\left(w^{(\beta)}_{n}\right)\right)\geq c, ~~\mathbb{Q}\text{-a.s.}.$$
    Using the conditional Borel-Cantelli lemma, this is equivalent to show that, for any $c\geq1$,
    $$\sum_{n=1}^{\infty}\mathbb{Q}\left.\left(U\left(\theta^{n}\xi,V\left(w^{(\beta)}_{n}\right)+\beta\right)e^{-V\left(w^{(\beta)}_{n}\right)}\tilde{X}\left(V\left(w^{(\beta)}_{n}\right)\right)\geq c~\right| \mathscr{G}^{(\beta)}_{n}\right)=\infty,~~\mathbb{Q}\text{-a.s.},$$
    where $\left\{\mathscr{G}^{(\beta)}_{n}\right\}$ is the filtration containing all the information of the spine and its siblings. By the spinal decomposition and the definition of $\mathbb{Q}$, we get, for any $c\geq1$,
    $$\sum_{n=1}^{\infty}\mathbb{E}\left.\left(\tilde{X}\left(\zeta^{(\beta)}_{n}\right)\mathbf{1}_{\left\{U\left(\theta^{n}\xi,\zeta^{(\beta)}_{n}+\beta\right)e^{-\zeta^{(\beta)}_{n}}\tilde{X}\left(\zeta^{(\beta)}_{n}\right)\geq c\right\}}~\right| \zeta^{(\beta)}_{n}\right)=\infty,~~\mathbb{P}\text{-a.s.}.$$
\end{proof}

\subsection{The sufficient condition}

In this subsection, we show that for all $\beta\geq0$, $D^{(\beta)}_{n}$ converges in $L^{1}\left(\mathbb{P}_{\xi}\right)$ to $D^{(\beta)}_{\infty}$ for almost all $\xi$ under the assumption (\ref{iff}).

\begin{lemma}\label{sufficiency}
	If $\left(\ref{iff}\right)$ holds, then for all $\beta\geq0$, $\mathbb{E}_{\xi}\left[D^{(\beta)}_{\infty}\right]=U(\xi,\beta), ~\mathbf{P}\text{-a.s.}$.
\end{lemma}

\begin{proof}
	According to Proposition \ref{BK th 2.1} $\left(1\right)$, it suffices to prove that
	\begin{equation}\label{suff 1}
		\begin{aligned}
			&\mathbb{E}\left[Y\log^2_{+}Y+Z\log_{+}Z\right]<\infty\\
			\Longrightarrow~~ &\sum_{n=1}^{\infty}\mathbb{E}_{\xi,\zeta^{(\beta)}_{n}}\left[\tilde{X}\left(\left(U(\theta^{n}\xi,\zeta^{(\beta)}_{n}+\beta)e^{-\zeta^{(\beta)}_{n}}\tilde{X}\right) \wedge 1\right)\right]U(\xi,\beta)<\infty, ~~\mathbb{P}\text{-a.s.}.
		\end{aligned}		
	\end{equation}
    By $\left(\ref{asymptotic behaviour of harmonic function 2}\right)$ and for almost all $\xi$, $\zeta^{(\beta)}_{n}+\beta\to\infty$ $\mathbb{P}_{\xi}\text{-a.s.}$ as $n\to\infty$, we have, as $n\to\infty$,
    \begin{equation}\nonumber
    	\begin{aligned}    			&\frac{\sum_{|u|=1}U\left(\theta^{n+1}\xi,\zeta^{(\beta)}_{n}+\Delta V(u)+\beta\right)e^{-\Delta V(u)}\mathbf{1}_{\left\{\zeta^{(\beta)}_{n}+\Delta V(u)\geq -\beta\right\}}}{U\left(\theta^{n}\xi,\zeta^{(\beta)}_{n}+\beta\right)}\\
    	\sim&\frac{\sum_{|u|=1}\left(\zeta^{(\beta)}_{n}+\Delta V(u)+\beta\right)e^{-\Delta V(u)}\mathbf{1}_{\left\{\zeta^{(\beta)}_{n}+\Delta V(u)\geq -\beta\right\}}}{U\left(\theta^{n}\xi,\zeta^{(\beta)}_{n}+\beta\right)}.
    	\end{aligned}
    \end{equation}
    Since $\left(\zeta^{(\beta)}_{n}+\Delta V(u)+\beta\right)\mathbf{1}_{\left\{\zeta^{(\beta)}_{n}+\Delta V(u)\geq -\beta\right\}}\leq \zeta^{(\beta)}_{n}+\beta+\Delta V(u)\mathbf{1}_{\left\{\Delta V(u)\geq 0\right\}}$, we obtain
    \begin{equation}\nonumber
    	\begin{aligned}
    		&\frac{\sum_{|u|=1}\left(\zeta^{(\beta)}_{n}+\Delta V(u)+\beta\right)e^{-\Delta V(u)}\mathbf{1}_{\left\{\zeta^{(\beta)}_{n}+\Delta V(u)\geq -\beta\right\}}}{U\left(\theta^{n}\xi,\zeta^{(\beta)}_{n}+\beta\right)}\\
    		\leq&\frac{\sum_{|u|=1}\left(\zeta^{(\beta)}_{n}+\beta\right)e^{-\Delta V(u)}}{U\left(\theta^{n}\xi,\zeta^{(\beta)}_{n}+\beta\right)}+\frac{\sum_{|u|=1}\Delta V(u)e^{-\Delta V(u)}\mathbf{1}_{\left\{\Delta V(u)\geq 0\right\}}}{U\left(\theta^{n}\xi,\zeta^{(\beta)}_{n}+\beta\right)}\\
    		=&:\frac{\left(\zeta^{(\beta)}_{n}+\beta\right)\tilde{Y}(\xi_{n+1})}{U\left(\theta^{n}\xi,\zeta^{(\beta)}_{n}+\beta\right)}+\frac{\tilde{Z}(\xi_{n+1})}{U\left(\theta^{n}\xi,\zeta^{(\beta)}_{n}+\beta\right)}\\
    		\leq&2\max\left\{\frac{\left(\zeta^{(\beta)}_{n}+\beta\right)\tilde{Y}(\xi_{n+1})}{U\left(\theta^{n}\xi,\zeta^{(\beta)}_{n}+\beta\right)},\frac{\tilde{Z}(\xi_{n+1})}{U\left(\theta^{n}\xi,\zeta^{(\beta)}_{n}+\beta\right)}\right\},
    	\end{aligned}
    \end{equation}
    where $\left(\tilde{Y}(\xi_{n+1}),\tilde{Z}(\xi_{n+1})\right)$ is independent of $\zeta^{(\beta)}_{n}$ under $\mathbb{P}_{\xi}$.
    
    Therefore, we only need to show that
    $$\mathbb{E}\left[Y\log^2_{+}Y+Z\log_{+}Z\right]<\infty$$
    \begin{equation}\label{suff 2}
    	\Longrightarrow~~
    	\begin{cases}
    		\mathbb{E}\left\{\sum_{n=1}^{\infty}\mathbb{E}_{\xi,\zeta^{(\beta)}_{n}}\left[\frac{\left(\zeta^{(\beta)}_{n}+\beta\right)\tilde{Y}(\xi_{n+1})}{U\left(\theta^{n}\xi,\zeta^{(\beta)}_{n}+\beta\right)}\left(\left(e^{-\zeta^{(\beta)}_{n}}\left(\zeta^{(\beta)}_{n}+\beta\right)\tilde{Y}(\xi_{n+1})\right) \wedge 1\right)\right]U(\xi,\beta)\right\}<\infty,\\
    		\\    		 	\mathbb{E}\left\{\sum_{n=1}^{\infty}\mathbb{E}_{\xi,\zeta^{(\beta)}_{n}}\left[\frac{\tilde{Z}(\xi_{n+1})}{U\left(\theta^{n}\xi,\zeta^{(\beta)}_{n}+\beta\right)}\left(\left(e^{-\zeta^{(\beta)}_{n}}\tilde{Z}(\xi_{n+1})\right) \wedge 1\right)\right]U(\xi,\beta)\right\}<\infty,
    	\end{cases}
    \end{equation}
    which implies $\left(\ref{suff 1}\right)$.
    
    To prove the first term on the right hand side of $\left(\ref{suff 2}\right)$, by the inequality $e^{x/2}\geq x$ for all $x>0$, we have
    \begin{equation}\nonumber
    	\begin{aligned}
    		&\mathbb{E}\left\{\sum_{n=1}^{\infty}\mathbb{E}_{\xi,\zeta^{(\beta)}_{n}}\left[\frac{\left(\zeta^{(\beta)}_{n}+\beta\right)\tilde{Y}(\xi_{n+1})}{U\left(\theta^{n}\xi,\zeta^{(\beta)}_{n}+\beta\right)}\left(\left(e^{-\zeta^{(\beta)}_{n}}\left(\zeta^{(\beta)}_{n}+\beta\right)\tilde{Y}(\xi_{n+1})\right) \wedge1\right)\right]U(\xi,\beta)\right\}\\
    		\leq&\mathbb{E}\left\{\sum_{n=1}^{\infty}\mathbb{E}_{\xi,\zeta^{(\beta)}_{n}}\left[\frac{\left(\zeta^{(\beta)}_{n}+\beta\right)\tilde{Y}(\xi_{n+1})}{U\left(\theta^{n}\xi,\zeta^{(\beta)}_{n}+\beta\right)}\left(\left(e^{-\zeta^{(\beta)}_{n}/2+\beta/2}\tilde{Y}(\xi_{n+1})\right) \wedge1\right)\right]U(\xi,\beta)\right\}\\
    		\leq&\mathbb{E}\left\{\sum_{n=1}^{\infty}\mathbb{E}_{\xi}\left[\frac{\left(\zeta^{(\beta)}_{n}+\beta\right)e^{-\zeta^{(\beta)}_{n}/2+\beta/2}\tilde{Y}^{2}(\xi_{n+1})}{U\left(\theta^{n}\xi,\zeta^{(\beta)}_{n}+\beta\right)}\mathbf{1}_{\left\{\zeta^{(\beta)}_{n}\geq2\log\tilde{Y}(\xi_{n+1})+\beta\right\}}\bigg | \zeta^{(\beta)}_{n}\right]U(\xi,\beta)\right\}\\
    		&+\mathbb{E}\left\{\sum_{n=1}^{\infty}\mathbb{E}_{\xi}\left[\frac{\left(\zeta^{(\beta)}_{n}+\beta\right)\tilde{Y}(\xi_{n+1})}{U\left(\theta^{n}\xi,\zeta^{(\beta)}_{n}+\beta\right)}\mathbf{1}_{\left\{\zeta^{(\beta)}_{n}<2\log\tilde{Y}(\xi_{n+1})+\beta\right\}}\bigg | \zeta^{(\beta)}_{n}\right]U(\xi,\beta)\right\}.
    	\end{aligned}
    \end{equation}
    Hence, it follows by Lemma \ref{lem the integral of series of zeta_{n}} and $\left(\ref{relation of renewal measure with L measure}\right)$ that
    \begin{equation}\nonumber
    	\begin{aligned}
    		&\mathbb{E}\left\{\sum_{n=1}^{\infty}\mathbb{E}_{\xi,\zeta^{(\beta)}_{n}}\left[\frac{\left(\zeta^{(\beta)}_{n}+\beta\right)\tilde{Y}(\xi_{n+1})}{U\left(\theta^{n}\xi,\zeta^{(\beta)}_{n}+\beta\right)}\left(\left(e^{-\zeta^{(\beta)}_{n}}\left(\zeta^{(\beta)}_{n}+\beta\right)\tilde{Y}(\xi_{n+1})\right) \wedge1\right)\right]U(\xi,\beta)\right\}\\
    		\leq&\int_{-\beta}^{\infty}\mathbb{E}\left[\left(x+\beta\right)e^{-x/2+\beta/2}\tilde{Y}^{2}(\xi_{n+1})\mathbf{1}_{\left\{x\geq2\log\tilde{Y}(\xi_{n+1})+\beta\right\}}\right]\mathcal{R}^{(\beta)}(dx)\\
    		&+\int_{-\beta}^{\infty}\mathbb{E}\left[\left(x+\beta\right)\tilde{Y}(\xi_{n+1})\mathbf{1}_{\left\{x<2\log\tilde{Y}(\xi_{n+1})+\beta\right\}}\right]\mathcal{R}^{(\beta)}(dx)\\
    		\leq&c_{2}\int_{0}^{\infty}\mathbb{E}\left[xe^{-x/2+\beta}\tilde{Y}^{2}(\xi_{n+1})\mathbf{1}_{\left\{x\geq2\log\tilde{Y}(\xi_{n+1})+2\beta\right\}}\right]dx\\
    		&+c_{2}\int_{0}^{\infty}\mathbb{E}\left[x\tilde{Y}(\xi_{n+1})\mathbf{1}_{\left\{x<2\log\tilde{Y}(\xi_{n+1})+2\beta\right\}}\right]dx\\
    		=&c_{2}e^{\beta}\mathbb{E}\left[\tilde{Y}^{2}(\xi_{n+1})\int_{2\left(\log\tilde{Y}(\xi_{n+1})+\beta\right)_{+}}^{\infty}xe^{-x/2}\,dx\right]+c_{2}\mathbb{E}\left[\tilde{Y}(\xi_{n+1})\int_{0}^{2\left(\log\tilde{Y}(\xi_{n+1})+\beta\right)_{+}}x\,dx\right]\\
    		\leq&c_{1}(\beta)\mathbb{E}\left(\tilde{Y}(\xi_{n+1})\log_{+}\tilde{Y}(\xi_{n+1})\right)+c_{2}(\beta)\mathbb{E}\left(\tilde{Y}(\xi_{n+1})\log^2_{+}\tilde{Y}(\xi_{n+1})\right)\\
    		=&c_{1}(\beta)\mathbb{E}\left(Y\log_{+}Y\right)+c_{2}(\beta)\mathbb{E}\left(Y\log^2_{+}Y\right)\\
    		<&\infty.
    	\end{aligned}
    \end{equation}
    
    For the second term on the right hand side of $\left(\ref{suff 2}\right)$, by the same argument, we have
    \begin{equation}\nonumber
    	\begin{aligned}
    		&\mathbb{E}\left\{\sum_{n=1}^{\infty}\mathbb{E}_{\xi,\zeta^{(\beta)}_{n}}\left[\frac{\tilde{Z}(\xi_{n+1})}{U\left(\theta^{n}\xi,\zeta^{(\beta)}_{n}+\beta\right)}\left(\left(e^{-\zeta^{(\beta)}_{n}}\tilde{Z}(\xi_{n+1})\right) \wedge 1\right)\right]U(\xi,\beta)\right\}\\
    		\leq&c_{3}(\beta)\mathbb{E}\left(\tilde{Z}(\xi_{n+1})\log_{+}\tilde{Z}(\xi_{n+1})\right)\\
    		=&c_{3}(\beta)\mathbb{E}\left(Z\log_{+}Z\right)\\
    		<&\infty.
    	\end{aligned}
    \end{equation}
    This completes the proof of $\left(\ref{suff 2}\right)$ and hence Lemma \ref{sufficiency} is now proved.
\end{proof}

\begin{proof}[Proof of the sufficient condition of Theorem \ref{main result} $\left(2\right)$]
	Assume that (\ref{iff}) holds, by Lemma \ref{sufficiency}, for all $\beta\geq0$, $D^{(\beta)}_{n}$ is $L^{1}\left(\mathbb{P}_{\xi}\right)$ convergence for almost all $\xi$. Therefore, in view of Lemma \ref{connection} $\left(2\right)$, we prove that $D_{\infty}$ is non-degenerate for almost all $\xi$, which completes the proof of sufficiency.
\end{proof}

\subsection{The necessary condition}

In this subsection, we show that $D^{(\beta)}_{\infty}=0, ~\mathbb{P}\text{-a.s.}$ for all $\beta\geq0$ when (\ref{iff}) does not hold.

\begin{lemma}\label{necessity}
	If $\left(\ref{iff}\right)$ does not hold, then for all $\beta\geq0$, $\mathbb{E}\left[D^{(\beta)}_{\infty}\right]=0$, which is equivalent to $D^{(\beta)}_{\infty}=0, ~\mathbb{P}_{\xi}\text{-a.s.}$ for almost all $\xi$.
\end{lemma}

\begin{proof}
	According to Proposition \ref{BK th 2.1} $\left(2\right)$, it suffices to prove that, for any $c\geq1$,
	\begin{equation}\label{nece}
		\begin{aligned}
			&\mathbb{E}\left[Y\log^2_{+}Y+Z\log_{+}Z\right]=\infty\\
			\Longrightarrow~&\sum_{n=1}^{\infty}\mathbb{E}_{\zeta^{(\beta)}_{n}}\left[\tilde{X}\mathbf{1}_{\left\{U(\theta^{n}\xi,\zeta^{(\beta)}_{n}+\beta)e^{-\zeta^{(\beta)}_{n}}\tilde{X}\geq c\right\}}\right]U(\xi,\beta)=\infty, ~\mathbb{P}\text{-a.s.}.
		\end{aligned}
	\end{equation}

    Following the idea of Chen \cite{Che15}, we divide the assumption on the left hand side of $\left(\ref{nece}\right)$ into three cases as follows:
    \begin{equation}\label{three cases}
    	\begin{cases}
    		(i)~\mathbb{E}\left[Y\log^2_{+}Y\right]=\infty,~~\mathbb{E}\left[Y\log_{+}Y\right]<\infty,\\
    		(ii)~\mathbb{E}\left[Y\log_{+}Y\right]=\infty,\\
    		(iii)~\mathbb{E}\left[Z\log_{+}Z\right]=\infty.
    	\end{cases}
    \end{equation}

    Note that under $\mathbb{P}_{\xi,\zeta^{(\beta)}_{n}}$, $\left(\Delta V(u):|u|=1\right)$ is distributed as $L_{n+1}$. For any $x\in\mathbb{R}$, we define
    $$Y_{+}\left(\xi_{n+1},x\right):=\sum_{|u|=1}e^{-\Delta V(u)}\mathbf{1}_{\left\{\Delta V(u)\geq -x\right\}},$$
    $$Y_{-}\left(\xi_{n+1},x\right):=\sum_{|u|=1}e^{-\Delta V(u)}\mathbf{1}_{\left\{\Delta V(u)<-x\right\}}.$$
    Observe that $\tilde{Y}\left(\xi_{n+1}\right)=Y_{+}\left(\xi_{n+1},x\right)+Y_{-}\left(\xi_{n+1},x\right)$ and $\left(Y_{+}\left(\xi_{n+1},x\right),Y_{-}\left(\xi_{n+1},x\right),x\in\mathbb{R}\right)$ is independent of $\zeta^{(\beta)}_{n}$ under $\mathbb{P}_{\xi}$.
    
    Firstly, we give the proof of $\left(\ref{nece}\right)$ under the assumption $(i)$ of $\left(\ref{three cases}\right)$. By $\left(\ref{tilde{X}}\right)$ and Proposition \ref{harmonic function} $(4)$, under $\mathbb{P}_{\xi,\zeta^{(\beta)}_{n}}$, we have
    \begin{equation}\nonumber
    	\begin{aligned}
    		\tilde{X}=&\frac{\sum_{|u|=1}U\left(\theta^{n+1}\xi,\zeta^{(\beta)}_{n}+\Delta V(u)+\beta\right)e^{-\Delta V(u)}\mathbf{1}_{\left\{\zeta^{(\beta)}_{n}+\Delta V(u)+\beta\geq 0\right\}}}{U\left(\theta^{n}\xi,\zeta^{(\beta)}_{n}+\beta\right)}\\
    		\geq&\frac{\sum_{|u|=1}\left(\zeta^{(\beta)}_{n}+\Delta V(u)+\beta\right)e^{-\Delta V(u)}\mathbf{1}_{\left\{\Delta V(u)\geq -\left(\zeta^{(\beta)}_{n}+\beta\right)/2\right\}}}{U\left(\theta^{n}\xi,\zeta^{(\beta)}_{n}+\beta\right)}\\
    		\geq&\frac{\sum_{|u|=1}\left(\zeta^{(\beta)}_{n}+\beta\right)e^{-\Delta V(u)}\mathbf{1}_{\left\{\Delta V(u)\geq -\left(\zeta^{(\beta)}_{n}+\beta\right)/2\right\}}}{2U\left(\theta^{n}\xi,\zeta^{(\beta)}_{n}+\beta\right)}\\
    		=&\frac{\zeta^{(\beta)}_{n}+\beta}{2U\left(\theta^{n}\xi,\zeta^{(\beta)}_{n}+\beta\right)}Y_{+}\left(\xi_{n+1},\frac{\zeta^{(\beta)}_{n}+\beta}{2}\right).
    	\end{aligned}
    \end{equation}
    Thus, we only need to show that, for any fixed $c\geq1$, $\mathbb{P}\text{-a.s.}$,
    \begin{equation}\nonumber
    	\sum_{n=1}^{\infty}\mathbb{E}_{\zeta^{(\beta)}_{n}}\left[\frac{\left(\zeta^{(\beta)}_{n}+\beta\right)Y_{+}\left(\xi_{n+1},(\zeta^{(\beta)}_{n}+\beta)/2\right)}{U\left(\theta^{n}\xi,\zeta^{(\beta)}_{n}+\beta\right)}\mathbf{1}_{\left\{e^{-\zeta^{(\beta)}_{n}}(\zeta^{(\beta)}_{n}+\beta)Y_{+}\left(\xi_{n+1},\frac{\zeta^{(\beta)}_{n}+\beta}{2}\right)\geq c\right\}}\right]U(\xi,\beta)=\infty.
    \end{equation}
    Since for almost all $\xi$, $\zeta^{(\beta)}_{n}+\beta\to\infty$ $\mathbb{P}_{\xi}\text{-a.s.}$ as $n\to\infty$ and $\left(\ref{asymptotic behaviour of harmonic function 2}\right)$, it suffices to prove that
    \begin{equation}\nonumber
    	\sum_{n=1}^{\infty}\mathbb{E}\left[Y_{+}\left(\xi_{n+1},\frac{\zeta^{(\beta)}_{n}+\beta}{2}\right)\mathbf{1}_{\left\{\log Y_{+}\left(\xi_{n+1},\frac{\zeta^{(\beta)}_{n}+\beta}{2}\right)\geq\zeta^{(\beta)}_{n}\right\}}\bigg| \zeta^{(\beta)}_{n}\right]U(\xi,\beta)=\infty,~~\mathbb{P}\text{-a.s.},
    \end{equation}
    which is
    \begin{equation}\label{nece 1}
    	\sum_{n=1}^{\infty}F\left(\frac{\zeta^{(\beta)}_{n}+\beta}{2},\zeta^{(\beta)}_{n}\right)U(\xi,\beta)=\infty, ~~\mathbb{P}\text{-a.s.},
    \end{equation}
    where
    $$F\left(x,y\right):=\mathbf{E}\left[F\left(\xi_{n+1},x,y\right)\right]~~\text{and}~~F\left(\xi_{n+1},x,y\right):=\mathbb{E}_{\xi}\left[Y_{+}\left(\xi_{n+1},x\right)\mathbf{1}_{\left\{\log Y_{+}\left(\xi_{n+1},x\right)\geq y\right\}}\right], ~~x,y\in\mathbb{R}.$$
    
    Let $$F_{1}\left(\xi_{n+1},y\right):=\mathbb{E}_{\xi}\left[\tilde{Y}\left(\xi_{n+1}\right)\mathbf{1}_{\left\{\log \tilde{Y}\left(\xi_{n+1}\right)\geq y\right\}}\right], ~~y\in\mathbb{R}$$ and $F_{1}\left(y\right):=\mathbf{E}\left[F_{1}\left(\xi_{n+1},y\right)\right]$. Note that, by the assumption $\left(\ref{boundary condition}\right)$,
    $$0\leq F\left(\xi_{n+1},x,y\right)\leq F_{1}\left(\xi_{n+1},y\right)\leq \mathbb{E}_{\xi}\left(\tilde{Y}\left(\xi_{n+1}\right)\right)=1, ~~\mathbf{P}\text{-a.s.}.$$
    It follows that $F_{1}\left(y\right)$ is a non-negative, non-increasing function. By the assumption $(i)$ of $\left(\ref{three cases}\right)$, we obtain
    \begin{equation}\nonumber
    	\begin{aligned}
    		\int_{-\beta}^{\infty}F_{1}\left(y\right)(y+\beta)\,dy&=\int_{-\beta}^{\infty}\mathbb{E}\left[\tilde{Y}\left(\xi_{n+1}\right)\mathbf{1}_{\left\{\log \tilde{Y}\left(\xi_{n+1}\right)\geq y\right\}}\right](y+\beta)\,dy\\
    		&=\mathbb{E}\left[\tilde{Y}\left(\xi_{n+1}\right)\int_{-\beta}^{\log_{+}\tilde{Y}\left(\xi_{n+1}\right)}(y+\beta)\,dy\right]\\
    		&=\frac{\mathbb{E}\left[\tilde{Y}\left(\xi_{n+1}\right)\left(\log_{+}\tilde{Y}\left(\xi_{n+1}\right)+\beta\right)^2\right]}{2}\\
    		&=\frac{\mathbb{E}\left[Y\left(\log_{+}Y+\beta\right)^2\right]}{2}\\
    		&=\infty.
    	\end{aligned}
    \end{equation}
    By Proposition \ref{prop 0-1 law}, we have
    \begin{equation}\label{nece 1.1}
    	\sum_{n=1}^{\infty}F_{1}\left(\zeta^{(\beta)}_{n}\right)U(\xi,\beta)=\infty, ~~\mathbb{P}\text{-a.s.}.
    \end{equation}

    We are left to prove that
    \begin{equation}\nonumber
    	\sum_{n=1}^{\infty}\left[F_{1}\left(\zeta^{(\beta)}_{n}\right)-F\left(\frac{\zeta^{(\beta)}_{n}+\beta}{2},\zeta^{(\beta)}_{n}\right)\right]U(\xi,\beta)<\infty, ~~\mathbb{P}\text{-a.s.},
    \end{equation}
    which, together with $\left(\ref{nece 1.1}\right)$, implies $\left(\ref{nece 1}\right)$. Equivalently, we only need to prove that
    \begin{equation}\label{nece 1.2}
    	\sum_{n=1}^{\infty}\frac{\zeta^{(\beta)}_{n}+\beta}{U\left(\theta^{n}\xi,\zeta^{(\beta)}_{n}+\beta\right)}\left[F_{1}\left(\zeta^{(\beta)}_{n}\right)-F\left(\frac{\zeta^{(\beta)}_{n}+\beta}{2},\zeta^{(\beta)}_{n}\right)\right]U(\xi,\beta)<\infty, ~~\mathbb{P}\text{-a.s.},
    \end{equation}
    We begin our proof by giving an upper bound for $F_{1}\left(\xi_{n+1},y\right)-F\left(\xi_{n+1},x,y\right)$. For any $x,y\in\mathbb{R}$, we obtain
    \begin{equation}\nonumber
    	\begin{aligned}
    		&F_{1}\left(\xi_{n+1},y\right)-F\left(\xi_{n+1},x,y\right)\\
    		=&\mathbb{E}_{\xi}\left[\tilde{Y}\left(\xi_{n+1}\right)\mathbf{1}_{\left\{\log \tilde{Y}\left(\xi_{n+1}\right)\geq y\right\}}\right]-\mathbb{E}_{\xi}\left[Y_{+}\left(\xi_{n+1},x\right)\mathbf{1}_{\left\{\log Y_{+}\left(\xi_{n+1},x\right)\geq y\right\}}\right]\\
    		=&\mathbb{E}_{\xi}\left[\tilde{Y}\left(\xi_{n+1}\right)\mathbf{1}_{\left\{\log \tilde{Y}\left(\xi_{n+1}\right)\geq y>\log Y_{+}\left(\xi_{n+1},x\right)\right\}}\right]+\mathbb{E}_{\xi}\left[\tilde{Y}\left(\xi_{n+1}\right)\mathbf{1}_{\left\{\log Y_{+}\left(\xi_{n+1},x\right)\geq y\right\}}\right]\\
    		&-\mathbb{E}_{\xi}\left[Y_{+}\left(\xi_{n+1},x\right)\mathbf{1}_{\left\{\log Y_{+}\left(\xi_{n+1},x\right)\geq y\right\}}\right]\\
    		=&\mathbb{E}_{\xi}\left[\tilde{Y}\left(\xi_{n+1}\right)\mathbf{1}_{\left\{\log \tilde{Y}\left(\xi_{n+1}\right)\geq y>\log Y_{+}\left(\xi_{n+1},x\right)\right\}}\right]+\mathbb{E}_{\xi}\left[Y_{-}\left(\xi_{n+1},x\right)\mathbf{1}_{\left\{\log Y_{+}\left(\xi_{n+1},x\right)\geq y\right\}}\right]\\
    		\leq&\mathbb{E}_{\xi}\left[2Y_{-}\left(\xi_{n+1},x\right)\mathbf{1}_{\left\{\log \tilde{Y}\left(\xi_{n+1}\right)\geq y>\log Y_{+}\left(\xi_{n+1},x\right),~Y_{-}\left(\xi_{n+1},x\right)\geq Y_{+}\left(\xi_{n+1},x\right)\right\}}\right]\\
    		&+\mathbb{E}_{\xi}\left[\tilde{Y}\left(\xi_{n+1}\right)\mathbf{1}_{\left\{\log \tilde{Y}\left(\xi_{n+1}\right)\geq y>\log Y_{+}\left(\xi_{n+1},x\right),~Y_{-}\left(\xi_{n+1},x\right)<Y_{+}\left(\xi_{n+1},x\right)\right\}}\right]\\
    		&+\mathbb{E}_{\xi}\left[Y_{-}\left(\xi_{n+1},x\right)\mathbf{1}_{\left\{\log Y_{+}\left(\xi_{n+1},x\right)\geq y\right\}}\right]\\
    		\leq&3\mathbb{E}_{\xi}\left[Y_{-}\left(\xi_{n+1},x\right)\right]+\mathbb{E}_{\xi}\left[\tilde{Y}\left(\xi_{n+1}\right)\mathbf{1}_{\left\{\log \tilde{Y}\left(\xi_{n+1}\right)\geq y>\log Y_{+}\left(\xi_{n+1},x\right),~Y_{-}\left(\xi_{n+1},x\right)<Y_{+}\left(\xi_{n+1},x\right)\right\}}\right]\\
    		\leq&3\mathbb{E}_{\xi}\left[Y_{-}\left(\xi_{n+1},x\right)\right]+\mathbb{E}_{\xi}\left[\tilde{Y}\left(\xi_{n+1}\right)\mathbf{1}_{\left\{\log \tilde{Y}\left(\xi_{n+1}\right)\geq y>\log \left(\tilde{Y}\left(\xi_{n+1}\right)/2\right)\right\}}\right]\\
    		=&:3A_{1}\left(\xi_{n+1},x\right)+A_{2}\left(\xi_{n+1},y\right),
    	\end{aligned}
    \end{equation}
    where the first and last inequality follow from  $\tilde{Y}\left(\xi_{n+1}\right)\leq2\max\left\{Y_{+}\left(\xi_{n+1},x\right),Y_{-}\left(\xi_{n+1},x\right)\right\}$.
    Therefore,
    \begin{equation}\nonumber
    	\begin{aligned}
    		&\sum_{n=1}^{\infty}\frac{\zeta^{(\beta)}_{n}+\beta}{U\left(\theta^{n}\xi,\zeta^{(\beta)}_{n}+\beta\right)}\left[F_{1}\left(\zeta^{(\beta)}_{n}\right)-F\left(\frac{\zeta^{(\beta)}_{n}+\beta}{2},\zeta^{(\beta)}_{n}\right)\right]U(\xi,\beta)\\
    		\leq&\sum_{n=1}^{\infty}\frac{\zeta^{(\beta)}_{n}+\beta}{U\left(\theta^{n}\xi,\zeta^{(\beta)}_{n}+\beta\right)}\mathbf{E}\left[3A_{1}\left(\xi_{n+1},\frac{\zeta^{(\beta)}_{n}+\beta}{2}\right)+A_{2}\left(\xi_{n+1},\zeta^{(\beta)}_{n}\right)\right]U(\xi,\beta).
    	\end{aligned}
    \end{equation}
    Taking the expectation on both sides, we have
    \begin{equation}\nonumber
    	\begin{aligned}
    		&\mathbb{E}\left\{\sum_{n=1}^{\infty}\frac{\zeta^{(\beta)}_{n}+\beta}{U\left(\theta^{n}\xi,\zeta^{(\beta)}_{n}+\beta\right)}\left[F_{1}\left(\zeta^{(\beta)}_{n}\right)-F\left(\frac{\zeta^{(\beta)}_{n}+\beta}{2},\zeta^{(\beta)}_{n}\right)\right]U(\xi,\beta)\right\}\\
    		\leq&\mathbb{E}\left\{\sum_{n=1}^{\infty}\frac{\zeta^{(\beta)}_{n}+\beta}{U\left(\theta^{n}\xi,\zeta^{(\beta)}_{n}+\beta\right)}\mathbf{E}\left[3A_{1}\left(\xi_{n+1},\frac{\zeta^{(\beta)}_{n}+\beta}{2}\right)+A_{2}\left(\xi_{n+1},\zeta^{(\beta)}_{n}\right)\right]U(\xi,\beta)\right\}.
    	\end{aligned}
    \end{equation}
    Then, using Lemma \ref{lem the integral of series of zeta_{n}}, we deduce
    \begin{equation}\nonumber
    	\begin{aligned}
    		&\mathbb{E}\left\{\sum_{n=1}^{\infty}\frac{\zeta^{(\beta)}_{n}+\beta}{U\left(\theta^{n}\xi,\zeta^{(\beta)}_{n}+\beta\right)}\left[F_{1}\left(\zeta^{(\beta)}_{n}\right)-F\left(\frac{\zeta^{(\beta)}_{n}+\beta}{2},\zeta^{(\beta)}_{n}\right)\right]U(\xi,\beta)\right\}\\
    		\leq&\int_{-\beta}^{\infty}3(x+\beta)\mathbf{E}\left[A_{1}\left(\xi_{n+1},\frac{x+\beta}{2}\right)\right]\mathcal{R}^{(\beta)}(dx)+\int_{-\beta}^{\infty}(x+\beta)\mathbf{E}\left[A_{2}\left(\xi_{n+1},x\right)\right]\mathcal{R}^{(\beta)}(dx).
    	\end{aligned}
    \end{equation}

    We now aim to prove that the value of two integrals in the last equality are finite, which completes the proof of $\left(\ref{nece 1.2}\right)$. For the first term, by the many-to-one lemma, we have
    \begin{equation}\label{A_1(x)}
    	\mathbf{E}\left[A_{1}\left(\xi_{n+1},x\right)\right]=\mathbb{E}\left[Y_{-}\left(\xi_{n+1},x\right)\right]=\mathbb{E}\left[\sum_{|u|=1}e^{-\Delta V(u)}\mathbf{1}_{\left\{\Delta V(u)<-x\right\}}\right]=\mathbb{P}\left(S_{1}<-x\right).
    \end{equation}
    Thus, from $\left(\ref{A_1(x)}\right)$, $\left(\ref{relation of renewal measure with L measure}\right)$ and the assumption $\left(\ref{2 more moment}\right)$, we obtain
    \begin{equation}\nonumber
    	\begin{aligned}
    		\int_{-\beta}^{\infty}(x+\beta)\mathbf{E}\left[A_{1}\left(\xi_{n+1},\frac{x+\beta}{2}\right)\right]\mathcal{R}^{(\beta)}(dx)=&\int_{-\beta}^{\infty}(x+\beta)\mathbb{P}\left(S_{1}<-\frac{x+\beta}{2}\right)\mathcal{R}^{(\beta)}(dx)\\
    		\leq&c_{2}\int_{0}^{\infty}x\mathbb{P}\left(S_{1}<-\frac{x}{2}\right)dx\\
    		=&c_{2}\mathbb{E}\left[\int_{0}^{2(-S_{1})_{+}}x\,dx\right]\\
    		=&2c_{2}\mathbb{E}\left[\left((-S_{1})_{+}\right)^{2}\right]\\
    		<&\infty.
    	\end{aligned}
    \end{equation}
    For the second term, by $\left(\ref{relation of renewal measure with L measure}\right)$ and the assumption $(i)$ of $\left(\ref{three cases}\right)$, we have
    \begin{equation}\nonumber
    	\begin{aligned}
    		&\int_{-\beta}^{\infty}(x+\beta)\mathbf{E}\left[A_{2}\left(\xi_{n+1},x\right)\right]\mathcal{R}^{(\beta)}(dx)\\
    		=&\int_{-\beta}^{\infty}(x+\beta)\mathbb{E}\left[\tilde{Y}\left(\xi_{n+1}\right)\mathbf{1}_{\left\{\log \tilde{Y}\left(\xi_{n+1}\right)\geq x>\log \left(\tilde{Y}\left(\xi_{n+1}\right)/2\right)\right\}}\right]\mathcal{R}^{(\beta)}(dx)\\
    		\leq&c_{2}\int_{0}^{\infty}x\mathbb{E}\left[\tilde{Y}\left(\xi_{n+1}\right)\mathbf{1}_{\left\{\log \tilde{Y}\left(\xi_{n+1}\right)+\beta \geq x>\log \left(\tilde{Y}\left(\xi_{n+1}\right)/2\right)+\beta\right\}}\right]dx\\
    		=&c_{2}\mathbb{E}\left[\tilde{Y}\left(\xi_{n+1}\right)\int_{\left(\log\left(\tilde{Y}(\xi_{n+1})/2\right)+\beta\right)_{+}}^{\left(\log\tilde{Y}(\xi_{n+1})+\beta\right)_{+}}x\,dx\right]\\
    		\leq&c_{4}(\beta)\mathbb{E}\left[Y\log_{+}Y\right]\\
    		<&\infty.
    	\end{aligned}
    \end{equation}
    We conclude that $\left(\ref{nece}\right)$ holds for the first case in $\left(\ref{three cases}\right)$.
    
    Secondly, we give the proof of $\left(\ref{nece}\right)$ under the assumption $(ii)$ of $\left(\ref{three cases}\right)$. By $\left(\ref{tilde{X}}\right)$ and Proposition \ref{harmonic function} $(4)$, under $\mathbb{P}_{\xi,\zeta^{(\beta)}_{n}}$, we get
    \begin{equation}\nonumber
    	\begin{aligned}
    		\tilde{X}=&\frac{\sum_{|u|=1}U\left(\theta^{n+1}\xi,\zeta^{(\beta)}_{n}+\Delta V(u)+\beta\right)e^{-\Delta V(u)}\mathbf{1}_{\left\{\zeta^{(\beta)}_{n}+\Delta V(u)+\beta\geq 0\right\}}}{U\left(\theta^{n}\xi,\zeta^{(\beta)}_{n}+\beta\right)}\\
    		\geq&\frac{\sum_{|u|=1}\left(\zeta^{(\beta)}_{n}+\Delta V(u)+\beta\right)e^{-\Delta V(u)}\mathbf{1}_{\left\{\zeta^{(\beta)}_{n}+\Delta V(u)+\beta\geq 1\right\}}}{U\left(\theta^{n}\xi,\zeta^{(\beta)}_{n}+\beta\right)}\\
    		\geq&\frac{\sum_{|u|=1}e^{-\Delta V(u)}\mathbf{1}_{\left\{\Delta V(u)\geq -\left(\zeta^{(\beta)}_{n}+\beta-1\right)\right\}}}{U\left(\theta^{n}\xi,\zeta^{(\beta)}_{n}+\beta\right)}\\
    		=&\frac{Y_{+}\left(\xi_{n+1},\zeta^{(\beta)}_{n}+\beta-1\right)}{U\left(\theta^{n}\xi,\zeta^{(\beta)}_{n}+\beta\right)}.
    	\end{aligned}
    \end{equation}
    Hence, it suffices to prove that, for any fixed $c\geq1$,
    \begin{equation}\nonumber
    	\sum_{n=1}^{\infty}\mathbb{E}_{\zeta^{(\beta)}_{n}}\left[\frac{Y_{+}\left(\xi_{n+1},\zeta^{(\beta)}_{n}+\beta-1\right)}{U\left(\theta^{n}\xi,\zeta^{(\beta)}_{n}+\beta\right)}\mathbf{1}_{\left\{e^{-\zeta^{(\beta)}_{n}}Y_{+}\left(\xi_{n+1},\zeta^{(\beta)}_{n}+\beta-1\right)\geq c\right\}}\right]U(\xi,\beta)=\infty, ~~\mathbb{P}\text{-a.s.}.
    \end{equation}
    Since for almost all $\xi$, $\zeta^{(\beta)}_{n}+\beta\to\infty$ $\mathbb{P}_{\xi}\text{-a.s.}$ as $n\to\infty$ and $\left(\ref{asymptotic behaviour of harmonic function 2}\right)$, this is equivalent to
    \begin{equation}\nonumber
    	\sum_{n=1}^{\infty}\mathbb{E}\left[\frac{Y_{+}\left(\xi_{n+1},\zeta^{(\beta)}_{n}+\beta-1\right)}{\zeta^{(\beta)}_{n}+\beta+1}\mathbf{1}_{\left\{\log Y_{+}\left(\xi_{n+1},\zeta^{(\beta)}_{n}+\beta-1\right)\geq\log c+\zeta^{(\beta)}_{n}\right\}}\bigg| \zeta^{(\beta)}_{n}\right]U(\xi,\beta)=\infty, ~~\mathbb{P}\text{-a.s.},
    \end{equation}
    which is
    \begin{equation}\label{nece 2}
    	\sum_{n=1}^{\infty}\frac{F\left(\zeta^{(\beta)}_{n}+\beta-1,\log c+\zeta^{(\beta)}_{n}\right)}{\zeta^{(\beta)}_{n}+\beta+1}U(\xi,\beta)=\infty, ~~\mathbb{P}\text{-a.s.},
    \end{equation}
    recalling that $F\left(x,\log c+y\right)=\mathbb{E}\left[Y_{+}\left(\xi_{n+1},x\right)\mathbf{1}_{\left\{\log Y_{+}\left(\xi_{n+1},x\right)\geq \log c+y\right\}}\right]$, $x,y\in\mathbb{R}$.
    
    Let $$F_{2}\left(\xi_{n+1},y\right):=\frac{\mathbb{E}_{\xi}\left[\tilde{Y}\left(\xi_{n+1}\right)\mathbf{1}_{\left\{\log \tilde{Y}\left(\xi_{n+1}\right)\geq \log c+y\right\}}\right]}{y+\beta+1}=\frac{F_{1}\left(\xi_{n+1},\log c+y\right)}{y+\beta+1}, ~~y\geq -\beta,$$ and $F_{2}\left(y\right):=\mathbf{E}\left[F_{2}\left(\xi_{n+1},y\right)\right]$. Clearly, $F_{2}\left(y\right)$ is non-increasing and
    $$0\leq F_{2}\left(y\right)=\mathbf{E}\left[\frac{F_{1}\left(\xi_{n+1},\log c+y\right)}{y+\beta+1}\right]=\frac{F_{1}\left(\log c+y\right)}{y+\beta+1}\leq1.$$
    By the assumption $(ii)$ of $\left(\ref{three cases}\right)$, we have
    \begin{equation}\nonumber
    	\begin{aligned}
    		\int_{-\beta}^{\infty}F_{2}\left(y\right)(y+\beta)\,dy&=\int_{-\beta}^{\infty}\frac{F_{1}\left(\log c+y\right)}{y+\beta+1}(y+\beta)\,dy\\
    		&\geq \int_{1}^{\infty}\frac{F_{1}\left(\log c+y\right)}{y+\beta+1}(y+\beta)\,dy\\
    		&\geq \frac{1}{2}\int_{1}^{\infty}F_{1}\left(\log c+y\right)\,dy\\
    		&=\frac{1}{2}\int_{1}^{\infty}\mathbb{E}\left[\tilde{Y}\left(\xi_{n+1}\right)\mathbf{1}_{\left\{\log \tilde{Y}\left(\xi_{n+1}\right)\geq \log c+y\right\}}\right]\,dy\\
    		&=\frac{1}{2}\mathbb{E}\left[\tilde{Y}\left(\xi_{n+1}\right)\left(\log(\tilde{Y}\left(\xi_{n+1}\right)/c)-1\right)_{+}\right]\\
    		&=\infty,
    	\end{aligned}
    \end{equation}
    which implies $\int_{-\beta}^{\infty}F_{2}\left(y\right)(y+\beta)\,dy=\infty$. By Proposition \ref{prop 0-1 law}, we get
    \begin{equation}\label{nece 2.1}
    	\sum_{n=1}^{\infty}\frac{F_{1}\left(\log c+\zeta^{(\beta)}_{n}\right)}{\zeta^{(\beta)}_{n}+\beta+1}U(\xi,\beta)=\infty, ~~\mathbb{P}\text{-a.s.}.
    \end{equation}
    
    We are left to prove that
    \begin{equation}\nonumber
    	\sum_{n=1}^{\infty}\frac{F_{1}\left(\log c+\zeta^{(\beta)}_{n}\right)-F\left(\zeta^{(\beta)}_{n}+\beta-1,\log c+\zeta^{(\beta)}_{n}\right)}{\zeta^{(\beta)}_{n}+\beta+1}U(\xi,\beta)<\infty, ~~\mathbb{P}\text{-a.s.},
    \end{equation}
    which, combined with $\left(\ref{nece 2.1}\right)$, implies $\left(\ref{nece 2}\right)$. Equivalently, we only need to prove that
    \begin{equation}\label{nece 2.2}
    	\sum_{n=1}^{\infty}\frac{F_{1}\left(\log c+\zeta^{(\beta)}_{n}\right)-F\left(\zeta^{(\beta)}_{n}+\beta-1,\log c+\zeta^{(\beta)}_{n}\right)}{U\left(\theta^{n}\xi,\zeta^{(\beta)}_{n}+\beta\right)}U(\xi,\beta)<\infty, ~~\mathbb{P}\text{-a.s.},
    \end{equation}
    By the same argument as the proof of first part, we have
    \begin{equation}\nonumber
    	\begin{aligned}
    		&\sum_{n=1}^{\infty}\frac{F_{1}\left(\log c+\zeta^{(\beta)}_{n}\right)-F\left(\zeta^{(\beta)}_{n}+\beta-1,\log c+\zeta^{(\beta)}_{n}\right)}{U\left(\theta^{n}\xi,\zeta^{(\beta)}_{n}+\beta\right)}U(\xi,\beta)\\
    		\leq&\sum_{n=1}^{\infty}\frac{\mathbf{E}\left[3A_{1}\left(\xi_{n+1},\zeta^{(\beta)}_{n}+\beta-1\right)+A_{2}\left(\xi_{n+1},\log c+\zeta^{(\beta)}_{n}\right)\right]}{U\left(\theta^{n}\xi,\zeta^{(\beta)}_{n}+\beta\right)}U(\xi,\beta).
    	\end{aligned}
    \end{equation}
    Taking the expectation on both sides, we get
    \begin{equation}\nonumber
    	\begin{aligned}
    		&\mathbb{E}\left\{\sum_{n=1}^{\infty}\frac{F_{1}\left(\log c+\zeta^{(\beta)}_{n}\right)-F\left(\zeta^{(\beta)}_{n}+\beta-1,\log c+\zeta^{(\beta)}_{n}\right)}{U\left(\theta^{n}\xi,\zeta^{(\beta)}_{n}+\beta\right)}U(\xi,\beta)\right\}\\
    		\leq&\mathbb{E}\left\{\sum_{n=1}^{\infty}\frac{\mathbf{E}\left[3A_{1}\left(\xi_{n+1},\zeta^{(\beta)}_{n}+\beta-1\right)+A_{2}\left(\xi_{n+1},\log c+\zeta^{(\beta)}_{n}\right)\right]}{U\left(\theta^{n}\xi,\zeta^{(\beta)}_{n}+\beta\right)}U(\xi,\beta)\right\}
    	\end{aligned}
    \end{equation}
    Then, by Lemma \ref{lem the integral of series of zeta_{n}}, we have
    \begin{equation}\nonumber
    	\begin{aligned}
    		&\mathbb{E}\left\{\sum_{n=1}^{\infty}\frac{F_{1}\left(\xi_{n+1},\log c+\zeta^{(\beta)}_{n}\right)-F\left(\xi_{n+1},\zeta^{(\beta)}_{n}+\beta-1,\log c+\zeta^{(\beta)}_{n}\right)}{U\left(\theta^{n}\xi,\zeta^{(\beta)}_{n}+\beta\right)}U(\xi,\beta)\right\}\\
    		\leq&\int_{-\beta}^{\infty}\left[3\mathbf{E}\left(A_{1}\left(\xi_{n+1},x+\beta-1\right)\right)+\mathbf{E}\left(A_{2}\left(\xi_{n+1},\log c+x\right)\right)\right]\mathcal{R}^{(\beta)}(dx).
    	\end{aligned}
    \end{equation}
    
    We turn to prove the finiteness of the above two integrals, which completes the proof of $\left(\ref{nece 2.2}\right)$. For the first integral, by $\left(\ref{A_1(x)}\right)$ and $\left(\ref{relation of renewal measure with L measure}\right)$, we obtain
    \begin{equation}\nonumber
    	\begin{aligned}
    		\int_{-\beta}^{\infty}\mathbf{E}\left[A_{1}\left(\xi_{n+1},x+\beta-1\right)\right]\mathcal{R}^{(\beta)}(dx)=&\int_{-\beta}^{\infty}\mathbb{P}\left(S_{1}<-(x+\beta-1)\right)\mathcal{R}^{(\beta)}(dx)\\
    		\leq&c_{2}\int_{0}^{\infty}\mathbb{P}\left(S_{1}<-(x-1)\right)\,dx\\
    		=&c_{2}\mathbb{E}\left[\int_{0}^{(-S_{1}+1)_{+}}\,dx\right]\\
    		=&c_{2}\mathbb{E}\left[(-S_{1}+1)_{+}\right]\\
    		<&\infty.
    	\end{aligned}
    \end{equation}
    For the second integral, by $\left(\ref{relation of renewal measure with L measure}\right)$, we get
    \begin{equation}\nonumber
    	\begin{aligned}
    		&\int_{-\beta}^{\infty}\mathbf{E}\left[A_{2}\left(\xi_{n+1},\log c+x\right)\right]\mathcal{R}^{(\beta)}(dx)\\
    		=&\int_{-\beta}^{\infty}\mathbb{E}\left[\tilde{Y}\left(\xi_{n+1}\right)\mathbf{1}_{\left\{\log \tilde{Y}\left(\xi_{n+1}\right)\geq \log c+x>\log \left(\tilde{Y}\left(\xi_{n+1}\right)/2\right)\right\}}\right]\mathcal{R}^{(\beta)}(dx)\\
    		\leq&c_{2}\int_{0}^{\infty}\mathbb{E}\left[\tilde{Y}\left(\xi_{n+1}\right)\mathbf{1}_{\left\{\log \tilde{Y}\left(\xi_{n+1}\right)+\beta\geq \log c+x>\log \left(\tilde{Y}\left(\xi_{n+1}\right)/2\right)+\beta\right\}}\right]dx\\
    		=&c_{2}\mathbb{E}\left[\tilde{Y}\left(\xi_{n+1}\right)\int_{\left(\log\left(\tilde{Y}(\xi_{n+1})/2c\right)+\beta\right)_{+}}^{\left(\log\left(\tilde{Y}(\xi_{n+1})/c\right)+\beta\right)_{+}}\,dx\right]\\
    		\leq&c_{5}(\beta)\mathbb{E}\left[\tilde{Y}\left(\xi_{n+1}\right)\right]\\
    		<&\infty.
    	\end{aligned}
    \end{equation}
    We obtain that $\left(\ref{nece}\right)$ holds for the second case in $\left(\ref{three cases}\right)$.
    
    Finally, we give the proof of $\left(\ref{nece}\right)$ under the assumption $(iii)$ of $\left(\ref{three cases}\right)$. By $\left(\ref{tilde{X}}\right)$ and Proposition \ref{harmonic function} $(4)$, under $\mathbb{P}_{\xi,\zeta^{(\beta)}_{n}}$, we get
    \begin{equation}\nonumber
    	\begin{aligned}
    		\tilde{X}=&\frac{\sum_{|u|=1}U\left(\theta^{n+1}\xi,\zeta^{(\beta)}_{n}+\Delta V(u)+\beta\right)e^{-\Delta V(u)}\mathbf{1}_{\left\{\zeta^{(\beta)}_{n}+\Delta V(u)+\beta\geq 0\right\}}}{U\left(\theta^{n}\xi,\zeta^{(\beta)}_{n}+\beta\right)}\\
    		\geq&\frac{\sum_{|u|=1}\left(\zeta^{(\beta)}_{n}+\Delta V(u)+\beta\right)e^{-\Delta V(u)}\mathbf{1}_{\left\{\Delta V(u)\geq 0\right\}}}{U\left(\theta^{n}\xi,\zeta^{(\beta)}_{n}+\beta\right)}\\
    		\geq&\frac{\sum_{|u|=1}\Delta V(u)e^{-\Delta V(u)}\mathbf{1}_{\left\{\Delta V(u)\geq 0\right\}}}{U\left(\theta^{n}\xi,\zeta^{(\beta)}_{n}+\beta\right)}\\
    		=&\frac{\tilde{Z}\left(\xi_{n+1}\right)}{U\left(\theta^{n}\xi,\zeta^{(\beta)}_{n}+\beta\right)}.
    	\end{aligned}
    \end{equation}
    Hence, we just need to prove that, for any fixed $c\geq1$,
    \begin{equation}\nonumber
    	\sum_{n=1}^{\infty}\mathbb{E}\left[\frac{\tilde{Z}\left(\xi_{n+1}\right)}{U\left(\theta^{n}\xi,\zeta^{(\beta)}_{n}+\beta\right)}\mathbf{1}_{\left\{e^{-\zeta^{(\beta)}_{n}}\tilde{Z}\left(\xi_{n+1}\right)\geq c\right\}}\bigg| \zeta^{(\beta)}_{n}\right]U(\xi,\beta)=\infty, ~~\mathbb{P}\text{-a.s.},
    \end{equation}
    Since for almost all $\xi$, $\zeta^{(\beta)}_{n}+\beta\to\infty$ $\mathbb{P}_{\xi}\text{-a.s.}$ as $n\to\infty$ and $\left(\ref{asymptotic behaviour of harmonic function 2}\right)$, this is equivalent to
    \begin{equation}\nonumber
    	\sum_{n=1}^{\infty}\mathbb{E}\left[\frac{\tilde{Z}\left(\xi_{n+1}\right)}{\zeta^{(\beta)}_{n}+\beta+1}\mathbf{1}_{\left\{\log \tilde{Z}\left(\xi_{n+1}\right)\geq \log c+\zeta^{(\beta)}_{n}\right\}}\bigg| \zeta^{(\beta)}_{n}\right]U(\xi,\beta)=\infty, ~~\mathbb{P}\text{-a.s.},
    \end{equation}
    which is
    \begin{equation}\label{nece 3}
    	\sum_{n=1}^{\infty}F_{3}\left(\zeta^{(\beta)}_{n}\right)U(\xi,\beta)=\infty, ~~\mathbb{P}\text{-a.s.},
    \end{equation}
    where $$F_{3}\left(x\right):=\frac{\mathbb{E}\left[\tilde{Z}\left(\xi_{n+1}\right)\mathbf{1}_{\left\{\log \tilde{Z}\left(\xi_{n+1}\right)\geq \log c+x\right\}}\right]}{x+\beta+1}, ~~x\geq -\beta.$$
    It is obviously that $F_{3}(x)$ is non-increasing and
    $$0\leq F_{3}\left(x\right)=\frac{\mathbb{E}\left[\tilde{Z}\left(\xi_{n+1}\right)\mathbf{1}_{\left\{\log \tilde{Z}\left(\xi_{n+1}\right)\geq \log c+x\right\}}\right]}{x+\beta+1}\leq\mathbb{E}(Z)=\mathbb{E}\left[(S_{1})_{+}\right]<\infty.$$
    Observe that, by the assumption $(iii)$ of $\left(\ref{three cases}\right)$,
    \begin{equation}\nonumber
    	\begin{aligned}
    		\int_{-\beta}^{\infty}F_{3}\left(x\right)(x+\beta)\,dx&\geq\int_{1}^{\infty}\frac{\mathbb{E}\left[\tilde{Z}\left(\xi_{n+1}\right)\mathbf{1}_{\left\{\log \tilde{Z}\left(\xi_{n+1}\right)\geq\log c+x\right\}}\right]}{x+\beta+1}(x+\beta)\,dx\\
    		&\geq\frac{1}{2}\int_{1}^{\infty}\mathbb{E}\left[\tilde{Z}\left(\xi_{n+1}\right)\mathbf{1}_{\left\{\log \tilde{Z}\left(\xi_{n+1}\right)\geq \log c+x\right\}}\right]\,dx\\
    		&=\frac{1}{2}\mathbb{E}\left[\tilde{Z}\left(\xi_{n+1}\right)\left(\log(\tilde{Z}\left(\xi_{n+1}\right)/c)-1\right)_{+}\right]\\
    		&=\infty,
    	\end{aligned}
    \end{equation}
    By Proposition \ref{prop 0-1 law}, it follows that $\left(\ref{nece 3}\right)$ is valid. Therefore, we conclude that $\left(\ref{nece}\right)$ holds for the third case in $\left(\ref{three cases}\right)$.
\end{proof}

\begin{proof}[Proof of the necessary condition of Theorem \ref{main result} $\left(2\right)$]
	By Lemma \ref{necessity}, for almost all $\xi$, $D^{(\beta)}_{\infty}=0, ~\mathbb{P}_{\xi}\text{-a.s.}$ for all $\beta\geq0$ when (\ref{iff}) does not hold. Therefore, using Lemma \ref{connection} $\left(3\right)$, we obtain that $D_{\infty}$ is degenerate for almost all $\xi$, which completes the proof of necessity.
\end{proof}


\section*{Appendix}\label{section a}

\noindent{\it Proof of Lemma \ref{truncated}.}
	For any $v\in\mathbb{T} \backslash \left\{\varnothing\right\}$, we denote by $\overleftarrow{v}$ its parent.
	By the branching property, many-to-one lemma and quenched harmonic property, we have
	\begin{equation}\nonumber
		\begin{aligned}
			&\mathbb{E}_{\xi,a}\left[D^{(\beta)}_{n+1}\mid\mathscr{F}_{n}\right]\\
			=&\mathbb{E}_{\xi,a}\left[\sum_{|u|=n}~\sum_{|v|=n+1:\overleftarrow{v}=u}U\left(\theta^{n+1}\xi,V(v)+\beta\right)e^{-V(v)}\mathbf{1}_{\left\{\min_{0\leq k\leq n}V(u_{k})\geq-\beta\right\}}\mathbf{1}_{\left\{V(v)\geq-\beta\right\}}\bigg |\mathscr{F}_{n}\right]\\
			=&\sum_{|u|=n}\mathbf{1}_{\left\{\min_{0\leq k\leq n}V(u_{k})\geq-\beta\right\}} \mathbb{E}_{\xi,V(u)}\left[\sum_{|v|=1}U\left(\theta^{n+1}\xi,V(v)+\beta\right)e^{-V(v)}\mathbf{1}_{\left\{V(v)\geq-\beta\right\}}\right]\\
			=&\sum_{|u|=n}\mathbf{1}_{\left\{\min_{0\leq k\leq n}V(u_{k})\geq-\beta\right\}}e^{-V(u)} \mathbb{E}_{\xi,V(u)}\left[U\left(\theta^{n+1}\xi,S_{1}+\beta\right)\mathbf{1}_{\left\{S_{1}\geq-\beta\right\}}\right]\\
			=&\sum_{|u|=n}\mathbf{1}_{\left\{\min_{0\leq k\leq n}V(u_{k})\geq-\beta\right\}}e^{-V(u)}U\left(\theta^{n}\xi,V(u)+\beta\right)\\
			=&D^{(\beta)}_{n},
		\end{aligned}
	\end{equation}
    it follows that $\left(D^{(\beta)}_{n},n\geq0\right)$ is a non-negative martingale under $\mathbb{P}_{\xi,a}$ and $\mathbb{P}_{a}$.
    By the martingale convergence theorem, we have the a.s. convergence of $D^{(\beta)}_{n}$.
\qed

\

\noindent{\it Proof of Lemma \ref{connection}.}
	$\left(1\right)$ By Theorem 7.1 of Biggins and Kyprianou \cite{BK04}, we have $W_{n}\to0,~\mathbb{P}\text{-a.s.}$. Since $e^{-\inf_{|u|=n}V(u)}\leq W_{n}\to0$, it follows that
	$$\inf_{|u|=n}V(u)\to\infty,~~\inf_{u\in\mathbb{T}}V(u)>-\infty,~~\mathbb{P}\text{-a.s.}.$$
	Hence, for any $\epsilon>0$, there exists $\beta:=\beta(\epsilon)$ such that
	$$\mathbb{P}\left(\inf_{u\in\mathbb{T}}V(u)\geq-\beta\right)\geq1-\epsilon.$$
	On the one hand, by Lemma \ref{truncated},
	$$D^{(\beta)}_{n}=\sum_{|u|=n}U\left(\theta^{n}\xi,V(u)+\beta\right)e^{-V(u)}\mathbf{1}_{\left\{\min_{0\leq k\leq n}V(u_{k})\geq-\beta\right\}}\to D^{(\beta)}_{\infty},~~\mathbb{P}\text{-a.s.};$$
	on the other hand, on the event $\left\{\inf_{u\in\mathbb{T}}V(u)\geq-\beta\right\}$, by (\ref{asymptotic behaviour of harmonic function 2}),
	$$D^{(\beta)}_{n}=\sum_{|u|=n}U\left(\theta^{n}\xi,V(u)+\beta\right)e^{-V(u)}\sim\sum_{|u|=n}\left(V(u)+\beta\right)e^{-V(u)}=D_{n}+\beta W_{n},~~\mathbb{P}\text{-a.s.}.$$
	Since $W_{n}\to0,~\mathbb{P}\text{-a.s.}$, it follows that with probability at least $1-\epsilon$, $D_{n}$ converges to a non-negative finite limit. This yields the $\mathbb{P}\text{-a.s.}$ convergence of $D_{n}$ by letting $\beta\to\infty$.
	
	$\left(2\right)$ If there exists $\beta\geq0$ such that $D^{(\beta)}_{n}$ converges in $L^{1}\left(\mathbb{P}_{\xi}\right)$ for almost all $\xi$, then we have $\mathbb{E}_{\xi}\left(D^{(\beta)}_{\infty}\right)=U(\xi,\beta)>0,~\mathbf{P}\text{-a.s.}$, in particular, $\mathbb{P}_{\xi}\left(D^{(\beta)}_{\infty}>0\right)>0, ~\mathbf{P}\text{-a.s.}$
	Since $D^{(\beta)}_{n}$ is non-decreasing in $\beta$, we deduce by (1) that $\mathbb{P}_{\xi}\left(D_{\infty}>0\right)>0, ~\mathbf{P}\text{-a.s.}$.
	
	$\left(3\right)$ If for almost all $\xi$, $D^{(\beta)}_{\infty}=0, ~\mathbb{P}_{\xi}\text{-a.s.}$ for all $\beta\geq0$, then by (1) again, we have $\mathbb{P}_{\xi}\left(D_{\infty}=0\right)=1,~\mathbf{P}\text{-a.s.}$.
\qed

\

\noindent{\it Proof of Proposition \ref{spinal decomposition}.}
	To describe the probabilities $\mathbb{P}_{\xi,a}$, $\mathbb{Q}^{(\beta)}_{\xi,a}$ and $\hat{\mathbb{P}}^{(\beta)}_{\xi,a}$, we use the Ulam-Harris-Neveu notations to encode the genealogical tree $\mathbb{T}$ with $\mathcal{U}:=\cup_{k=1}^{\infty}\left(\mathbb{N}^*\right)^k\cup\left\{\varnothing\right\}$, where $\mathbb{N}^*:=\left\{1,2,\cdots\right\}$. Vertices of a tree are labelled by their line of descent. For example, the vertex $u=k_{1}\cdots k_{n}$ means the $k_{n}$-th child of $\cdots$ of the $k_{1}$-th child of the initial vertex $\varnothing$. For two strings $u$ and $v$, we write $uv$ for the concatenated string. We refer to Section 1.1 of Mallein \cite{Mal15} for the rigorous presentation of the time-inhomogeneous branching random walk. Let $\left(g_{u},u\in\mathcal{U}\right)$ be a family of non-negative measurable functions, by the standard argument for measure extension theorem, it suffices to prove that for any $n\geq0$,
	$$\mathbb{E}_{\hat{\mathbb{P}}^{(\beta)}_{\xi,a}}\left[\prod_{|u|\leq n}g_{u}\left(\xi,V(u)\right)\right]=\mathbb{E}_{\mathbb{Q}^{(\beta)}_{\xi,a}}\left[\prod_{|u|\leq n}g_{u}\left(\xi,V(u)\right)\right],$$
	where $\mathbb{E}_{\hat{\mathbb{P}}^{(\beta)}_{\xi,a}}$ and $\mathbb{E}_{\mathbb{Q}^{(\beta)}_{\xi,a}}$ denote the corresponding expectation of $\hat{\mathbb{P}}^{(\beta)}_{\xi,a}$ and $\mathbb{Q}^{(\beta)}_{\xi,a}$ respectively.
	That is, by $\left(\ref{definition of quenched Q}\right)$ (the definition of $\mathbb{Q}^{(\beta)}_{\xi,a}$),
	\begin{equation}\label{hat P}
		\mathbb{E}_{\hat{\mathbb{P}}^{(\beta)}_{\xi,a}}\left[\prod_{|u|\leq n}g_{u}\left(\xi,V(u)\right)\right]=\mathbb{E}_{\xi,a}\left[\frac{D^{(\beta)}_{n}}{U\left(\xi,a+\beta\right)e^{-a}}\prod_{|u|\leq n}g_{u}\left(\xi,V(u)\right)\right].
	\end{equation}

	Let us define
	$$D^{(\beta)}_{n}(v):=U\left(\theta^{n}\xi,V(v)+\beta\right)e^{-V(v)}\mathbf{1}_{\left\{\min_{0\leq k\leq n}V(v_{k})\geq-\beta\right\}}, ~~n\geq1,$$
	and $D^{(\beta)}_{0}(\varnothing):=U\left(\xi,a+\beta\right)e^{-a}$ if $V\left(\varnothing\right)=a$. Clearly, $D^{(\beta)}_{n}=\sum_{|v|=n}D^{(\beta)}_{n}(v)$, $n\geq0$.
	We claim that for any $v\in\mathcal{U}$ with $|v|=n$,
	\begin{equation}\label{hat P spinal}
		\mathbb{E}_{\hat{\mathbb{P}}^{(\beta)}_{\xi,a}}\left[\mathbf{1}_{\left\{w^{(\beta)}_{n}=v\right\}}\prod_{|u|\leq n}g_{u}\left(\xi,V(u)\right)\right]=\mathbb{E}_{\xi,a}\left[\frac{D^{(\beta)}_{n}(v)}{U\left(\xi,a+\beta\right)e^{-a}}\prod_{|u|\leq n}g_{u}\left(\xi,V(u)\right)\right],
	\end{equation}
	which implies $\left(\ref{hat P}\right)$ by summing over $|v|=n$.
	
	We turn to the proof of $\left(\ref{hat P spinal}\right)$. We introduce some notations for our statement. For any $u\in\mathcal{U}$, we denote by $\underrightarrow{u}$ its children and $\mathbb{T}_{u}$ the subtree rooted at it. We write $\left\{\varnothing\rightsquigarrow u\right\}:=\left\{\varnothing,u_{1},\cdots,u_{|u|}\right\}$ for the set of vertices in the unique shortest path connecting $\varnothing$ to $u$. Decomposing the product $\prod_{|u|\leq n}g_{u}\left(\xi,V(u)\right)$ along the path $\left\{\varnothing\rightsquigarrow v\right\}$, $\left(\ref{hat P spinal}\right)$ can be written as
	\begin{equation}\label{hat P spinal 2}
		\begin{aligned}
			&\mathbb{E}_{\hat{\mathbb{P}}^{(\beta)}_{\xi,a}}\left[\mathbf{1}_{\left\{w^{(\beta)}_{n}=v\right\}}\prod_{i=0}^{n}g_{v_{i}}\left(\xi,V(v_{i})\right)\prod_{u\in \underrightarrow{v_{i-1}} \backslash v_{i}}h_{u}\left(\xi,V(u)\right)\right]\\
			=&\mathbb{E}_{\xi,a}\left[\frac{D^{(\beta)}_{n}(v)}{U\left(\xi,a+\beta\right)e^{-a}}\prod_{i=0}^{n}g_{v_{i}}\left(\xi,V(v_{i})\right)\prod_{u\in \underrightarrow{v_{i-1}} \backslash v_{i}}h_{u}\left(\xi,V(u)\right)\right],
		\end{aligned}		
	\end{equation}
	where $h_{u}(\xi,\cdot):=\mathbb{E}_{\theta^{|u|}\xi}\left[\prod_{z\in\mathbb{T}_{u}}g_{uz}\left(\xi,\cdot+V(z)\right)\mathbf{1}_{\left\{|z|\leq n-|u|\right\}}\right]$ and $\underrightarrow{v_{i-1}} \backslash v_{i}$ means the set of the siblings of $v_{i}$.
	
	Now we prove $\left(\ref{hat P spinal 2}\right)$ by induction. This equation obviously holds for $n=0$. Assume that it holds for $n-1$, we need to show that it is true for $n$. Let us introduce the filtration $\mathscr{G}^{(\beta)}_{n}:=\sigma\left(w^{(\beta)}_{i},V\left(w^{(\beta)}_{i}\right),0\leq i\leq n\right)\vee \sigma\left(\underrightarrow{w^{(\beta)}_{i}},V\left(\underrightarrow{w^{(\beta)}_{i}}\right),0\leq i<n\right)$, the information of the spine and its siblings. By the construction of $\hat{\mathbb{P}}^{(\beta)}_{\xi,a}$, given that $\left\{w^{(\beta)}_{n-1}=v_{n-1}\right\}$, the probability of the event $\left\{w^{(\beta)}_{n}=v\right\}$ is $\frac{D^{(\beta)}_{n}(v)}{D^{(\beta)}_{n}(v)+\sum_{u\in \underrightarrow{v_{n-1}} \backslash v}D^{(\beta)}_{n}(u)}$, and the point process generated by $w^{(\beta)}_{n-1}=v_{n-1}$ under $\hat{\mathbb{P}}^{(\beta)}_{\xi,a}$ has Radon-Nikodym derivative $\frac{D^{(\beta)}_{n}(v)+\sum_{u\in \underrightarrow{v_{n-1}} \backslash v}D^{(\beta)}_{n}(u)}{D^{(\beta)}_{n-1}(v_{n-1})}$ with respect to the point process generated by $v_{n-1}$ under $\mathbb{P}_{\xi,a}$. As a result, for the $n$-th term $\mathbf{1}_{\left\{w^{(\beta)}_{n}=v\right\}}g_{v}\left(\xi,V(v)\right)\prod_{u\in \underrightarrow{v_{n-1}} \backslash v}h_{u}\left(\xi,V(u)\right)$ in the product inside the left hand side of $\left(\ref{hat P spinal 2}\right)$, conditioned on $\mathscr{G}^{(\beta)}_{n-1}$, we have
	\begin{equation}\nonumber
		\begin{aligned}
			&\mathbb{E}_{\hat{\mathbb{P}}^{(\beta)}_{\xi,a}}\left[\mathbf{1}_{\left\{w^{(\beta)}_{n}=v\right\}}g_{v}\left(\xi,V(v)\right)\prod_{u\in \underrightarrow{v_{n-1}} \backslash v}h_{u}\left(\xi,V(u)\right)\bigg| \mathscr{G}^{(\beta)}_{n-1}\right]\\
			=&\mathbf{1}_{\left\{w^{(\beta)}_{n-1}=v_{n-1}\right\}}\mathbb{E}_{\hat{\mathbb{P}}^{(\beta)}_{\xi,a}}\left[\frac{D^{(\beta)}_{n}(v)g_{v}\left(\xi,V(v)\right)}{D^{(\beta)}_{n}(v)+\sum_{u\in \underrightarrow{v_{n-1}} \backslash v}D^{(\beta)}_{n}(u)}\prod_{u\in \underrightarrow{v_{n-1}} \backslash v}h_{u}\left(\xi,V(u)\right)\bigg| \mathscr{G}^{(\beta)}_{n-1}\right]\\
			=&\mathbf{1}_{\left\{w^{(\beta)}_{n-1}=v_{n-1}\right\}}\mathbb{E}_{\hat{\mathbb{P}}^{(\beta)}_{\xi,a}}\left[\frac{D^{(\beta)}_{n}(v)g_{v}\left(\xi,V(v)\right)}{D^{(\beta)}_{n}(v)+\sum_{u\in \underrightarrow{v_{n-1}} \backslash v}D^{(\beta)}_{n}(u)}\prod_{u\in \underrightarrow{v_{n-1}} \backslash v}h_{u}\left(\xi,V(u)\right)\bigg| \left(w^{(\beta)}_{n-1},V\left(w^{(\beta)}_{n-1}\right)\right)\right]\\
			=&\mathbf{1}_{\left\{w^{(\beta)}_{n-1}=v_{n-1}\right\}}\mathbb{E}_{\xi,a}\left[\frac{D^{(\beta)}_{n}(v)g_{v}\left(\xi,V(v)\right)}{D^{(\beta)}_{n-1}(v_{n-1})}\prod_{u\in \underrightarrow{v_{n-1}} \backslash v}h_{u}\left(\xi,V(u)\right)\bigg| V\left(v_{n-1}\right)\right]\\
			=&:\mathbf{1}_{\left\{w^{(\beta)}_{n-1}=v_{n-1}\right\}}f\left(\xi,V\left(v_{n-1}\right)\right).
		\end{aligned}
	\end{equation}
	Then, it follows from the above expression and the inductive hypothesis that
	\begin{equation}\nonumber
		\begin{aligned}
			&\mathbb{E}_{\hat{\mathbb{P}}^{(\beta)}_{\xi,a}}\left[\mathbf{1}_{\left\{w^{(\beta)}_{n}=v\right\}}\prod_{i=0}^{n}g_{v_{i}}\left(\xi,V(v_{i})\right)\prod_{u\in \underrightarrow{v_{i-1}} \backslash v_{i}}h_{u}\left(\xi,V(u)\right)\right]\\
			=&\mathbb{E}_{\hat{\mathbb{P}}^{(\beta)}_{\xi,a}}\left[\mathbf{1}_{\left\{w^{(\beta)}_{n-1}=v_{n-1}\right\}}f\left(\xi,V\left(v_{n-1}\right)\right)\prod_{i=0}^{n-1}g_{v_{i}}\left(\xi.V(v_{i})\right)\prod_{u\in \underrightarrow{v_{i-1}} \backslash v_{i}}h_{u}\left(\xi,V(u)\right)\right]\\
			=&\mathbb{E}_{\xi,a}\left[\frac{D^{(\beta)}_{n-1}(v_{n-1})}{U\left(\xi,a+\beta\right)e^{-a}}f\left(\xi,V\left(v_{n-1}\right)\right)\prod_{i=0}^{n-1}g_{v_{i}}\left(\xi,V(v_{i})\right)\prod_{u\in \underrightarrow{v_{i-1}} \backslash v_{i}}h_{u}\left(\xi,V(u)\right)\right]\\
			=&\mathbb{E}_{\xi,a}\left[\frac{D^{(\beta)}_{n}(v)}{U\left(\xi,a+\beta\right)e^{-a}}\prod_{i=0}^{n}g_{v_{i}}\left(\xi,V(v_{i})\right)\prod_{u\in \underrightarrow{v_{i-1}} \backslash v_{i}}h_{u}\left(\xi,V(u)\right)\right].
		\end{aligned}
	\end{equation}
	This proves $\left(\ref{hat P spinal 2}\right)$ and hence completes the proof of Proposition \ref{spinal decomposition}.
\qed

\

\noindent{\it Proof of Proposition \ref{law of the spine}.}   
	$\left(1\right)$ Recalling $\left(\ref{hat P spinal}\right)$ in the proof of Proposition \ref{spinal decomposition} and $\left(\ref{definition of quenched Q}\right)$ (the definition of $\mathbb{Q}^{(\beta)}_{\xi,a}$), we have
	\begin{equation}\nonumber
		\begin{aligned}
			\mathbb{E}_{\mathbb{Q}^{(\beta)}_{\xi,a}}\left[\mathbf{1}_{\left\{w^{(\beta)}_{n}=v\right\}}\prod_{|u|\leq n}g_{u}\left(\xi,V(u)\right)\right]=&\mathbb{E}_{\xi,a}\left[\frac{D^{(\beta)}_{n}(v)}{U\left(\xi,a+\beta\right)e^{-a}}\prod_{|u|\leq n}g_{u}\left(\xi,V(u)\right)\right]\\
			=&\mathbb{E}_{\mathbb{Q}^{(\beta)}_{\xi,a}}\left[\frac{D^{(\beta)}_{n}(v)}{D^{(\beta)}_{n}}\prod_{|u|\leq n}g_{u}\left(\xi,V(u)\right)\right],
		\end{aligned}
	\end{equation}
	which implies
	$$\mathbb{Q}^{(\beta)}_{\xi,a}\left(w^{(\beta)}_{n}=v\mid \mathscr{F}_{n}\right)=\frac{D^{(\beta)}_{n}(v)}{D^{(\beta)}_{n}}.$$
	This proves the first part of the proposition by recalling the definition of $D^{(\beta)}_{n}(v)$.
	
	$\left(2\right)$ For all $n$ and any measurable function $f:\mathbb{R}^{n+1}\to\mathbb{R}_{+}$, by part $\left(1\right)$, we have
	\begin{equation}\nonumber
		\begin{aligned}
			&\mathbb{E}_{\mathbb{Q}^{(\beta)}_{\xi,a}}\left[f\left(V(w^{(\beta)}_{0}), \cdots, V(w^{(\beta)}_{n})\right)\right]\\
			=&\mathbb{E}_{\mathbb{Q}^{(\beta)}_{\xi,a}}\left[\sum_{|v|=n}f\left(V(v_{0}), \cdots, V(v_{n})\right)\mathbf{1}_{\left\{w^{(\beta)}_{n}=v\right\}}\right]\\
			=&\mathbb{E}_{\mathbb{Q}^{(\beta)}_{\xi,a}}\left[\sum_{|v|=n}f\left(V(v_{0}), \cdots, V(v_{n})\right)\frac{U\left(\theta^{n}\xi,V(v)+\beta\right)e^{-V(v)}\mathbf{1}_{\left\{\min_{0\leq k\leq n}V(v_{k})\geq-\beta\right\}}}{D^{(\beta)}_{n}}\right],
		\end{aligned}
	\end{equation}
    then, from the definition of $\mathbb{Q}^{(\beta)}_{\xi,a}$ and many-to-one lemma, we obtain
    \begin{equation}\nonumber
    	\begin{aligned}
    		&\mathbb{E}_{\mathbb{Q}^{(\beta)}_{\xi,a}}\left[\sum_{|v|=n}f\left(V(v_{0}), \cdots, V(v_{n})\right)\frac{U\left(\theta^{n}\xi,V(v)+\beta\right)e^{-V(v)}\mathbf{1}_{\left\{\min_{0\leq k\leq n}V(v_{k})\geq-\beta\right\}}}{D^{(\beta)}_{n}}\right]\\
    		=&\mathbb{E}_{\xi,a}\left[\sum_{|v|=n}f\left(V(v_{0}), \cdots, V(v_{n})\right)\frac{U\left(\theta^{n}\xi,V(v)+\beta\right)e^{-V(v)}\mathbf{1}_{\left\{\min_{0\leq k\leq n}V(v_{k})\geq-\beta\right\}}}{U\left(\xi,a+\beta\right)e^{-a}}\right]\\
    		=&\mathbb{E}_{\xi, a}\left[f\left(S_{0}, \cdots, S_{n}\right)\frac{U(\theta^{n}\xi,S_{n}+\beta)}{U(\xi,a+\beta)}\mathbf{1}_{\left\{\min_{0\leq k\leq n}S_{k}\geq-\beta\right\}}\right].
    	\end{aligned}
    \end{equation}
    This completes the second part of the proposition.
\qed


\end{document}